\documentclass[11pt]{article}
\usepackage[margin=2.25cm]{geometry}
\usepackage[all=3]{nowidow}
\usepackage[utf8]{inputenc}
\usepackage[T1]{fontenc}
\usepackage[american]{babel}
\usepackage{amsmath,amssymb,amsthm}
\usepackage{graphicx}
\usepackage{xcolor}
\usepackage{csquotes}

\usepackage[
  backend=biber,
  style=numeric-comp,
  sorting=nyt,
  natbib=true,
  maxbibnames=99,
  giveninits=true,
  isbn=false,
  url=false
]{biblatex}
\AtEveryBibitem{%
  \clearfield{issn}%
  \clearfield{isbn}%
  \clearfield{url}%
  \clearfield{urlyear}%
  \clearfield{urlmonth}%
  \clearfield{urlday}%
}
\IfFileExists{biblio.bib}
  {\addbibresource{biblio.bib}}
  {\addbibresource{../biblio.bib}}

\usepackage{hyperref}

\usepackage[mathlines]{lineno}
\usepackage{etoolbox} %

\newcommand*\linenomathpatch[1]{%
  \cspreto{#1}{\linenomath}%
  \cspreto{#1*}{\linenomath}%
  \csappto{end#1}{\endlinenomath}%
  \csappto{end#1*}{\endlinenomath}%
}
\newcommand*\linenomathpatchAMS[1]{%
  \cspreto{#1}{\linenomathAMS}%
  \cspreto{#1*}{\linenomathAMS}%
  \csappto{end#1}{\endlinenomath}%
  \csappto{end#1*}{\endlinenomath}%
}

\expandafter\ifx\linenomath\linenomathWithnumbers
  \let\linenomathAMS\linenomathWithnumbers
  \patchcmd\linenomathAMS{\advance\postdisplaypenalty\linenopenalty}{}{}{}
\else
  \let\linenomathAMS\linenomathNonumbers
\fi

\linenomathpatch{equation}
\linenomathpatchAMS{gather}
\linenomathpatchAMS{multline}
\linenomathpatchAMS{align}
\linenomathpatchAMS{alignat}
\linenomathpatchAMS{flalign}

\usepackage{bbm}
\hypersetup{
	colorlinks=true,
    linkcolor={red!50!black},
    citecolor={blue!50!black},
    urlcolor={blue!80!black},
	bookmarksopen=true,
	bookmarksnumbered,
	bookmarksopenlevel=2,
	bookmarksdepth=3
}
\usepackage[capitalise]{cleveref}
\RemoveFromHook{label}[firstaid/cleveref]

\crefname{property}{Property}{Properties}
\crefname{conjecture}{Conjecture}{Conjectures}

\usepackage{enumitem}
\usepackage{mathtools}
\usepackage{xspace}

\newcommand{\td}{\mathrm{td}}

\newcommand{\tw}{\mathrm{tw}}
\newcommand{\minmax}{\mathrm{minmax}}
\newcommand{\maxmin}{\mathrm{maxmin}}
\newcommand{\fvs}{\mathrm{fvs}}
\newcommand{\octt}{\mathrm{oct}}
\newcommand{\vc}{\mathrm{vc}}

\newcommand{\pw}{\mathrm{pw}}
\newcommand{\aw}{$\alpha\omega\text{esome}$\xspace} 
 
\newcommand{\s}{s}
\newcommand{\opt}{\mathrm{opt}}
\newcommand{\tin}{\alpha \textnormal{-} \tw}
\newcommand{\pin}{\alpha \textnormal{-} \pw}

\newcommand{\rhocm}{\mu_{\rho,c}}
\newcommand{\twcm}{\mu_{\tw,c}}

\usepackage{color}
\usepackage{xspace}
\usepackage[most]{tcolorbox}
\usepackage{tikz}
\usetikzlibrary{positioning,shapes,calc,hobby,decorations.pathmorphing,decorations.pathreplacing, arrows.meta}
\usepackage[normalem]{ulem}

\usepackage{microtype}

\theoremstyle{plain}
\newtheorem{theorem}{Theorem}[section]
\newtheorem{lemma}[theorem]{Lemma}
\newtheorem{proposition}[theorem]{Proposition}

\newtheorem{observation}[theorem]{Observation}
\newtheorem{corollary}[theorem]{Corollary}

\crefname{observation}{observation}{observations}
\Crefname{observation}{Observation}{Observations}
\crefmultiformat{conjecture}{Conjectures~#2#1#3}{ and~#2#1#3}{, #2#1#3}{, and~#2#1#3}

\theoremstyle{definition}

\newtheorem{remark}[theorem]{Remark}

\newtheorem{question}[theorem]{Question}

\theoremstyle{remark}

\numberwithin{equation}{section}

\begin{document}

\title{Awesome graph parameters}

\author{
Kenny Bešter Štorgel\thanks{Faculty of Information Studies, University of Novo mesto, Ljubljanska cesta 31a, p.p. 603, SI-8000, Novo mesto, Slovenia; FAMNIT, University of Primorska, Glagoljaška 8, SI-6000, Koper, Slovenia; \texttt{kennystorgel.research@gmail.com}}
\and
Clément Dallard\thanks{Department of Informatics, University of Fribourg, Boulevard de Pérolles 90, CH-1700, Fribourg, Switzerland; \texttt{clement.dallard@unifr.ch}}
\and
Vadim Lozin\thanks{Mathematics Institute, University of Warwick, CV4 7AL, Coventry, United Kingdom; \texttt{V.Lozin@warwick.ac.uk}}
\and
Martin Milanič\thanks{FAMNIT, University of Primorska, Glagoljaška 8, SI-6000, Koper, Slovenia; IAM, University of Primorska, Muzejski trg 2, SI-6000, Koper, Slovenia; \texttt{martin.milanic@upr.si}}
\and
Viktor Zamaraev\thanks{School of Computer Science and Informatics, University of Liverpool, L69 3DR, Liverpool, United Kingdom; \texttt{Viktor.Zamaraev@liverpool.ac.uk}}
}

\date{}
\maketitle

\begin{abstract}
For a graph $G$, we denote by $\alpha(G)$ the size of a maximum independent set and by $\omega(G)$ the size of a maximum clique in $G$. 
Our paper lies on the edge of two lines of research, related to $\alpha$ and $\omega$, respectively. 
One of them studies $\alpha$-variants of graph parameters, such as $\alpha$-treewidth or $\alpha$-degeneracy. 
The second line deals with graph classes where some parameters are bounded by a function of $\omega(G)$. 
A famous example of this type is the family of $\chi$-bounded classes, where the chromatic number $\chi(G)$ is bounded by a function of $\omega(G)$.

A Ramsey-type argument implies that if the $\alpha$-variant of a graph parameter~$\rho$ is bounded by a constant in a hereditary class~$\mathcal{G}$, then $\rho$ is bounded by a function of $\omega$ in~$\mathcal{G}$. 
If the reverse implication also holds, we say that $\rho$ is \emph{awesome}.
Otherwise, we say that $\rho$ is \emph{awful}.
In the present paper, we identify a number of awesome and awful graph parameters, derive some structural and algorithmic applications of awesomeness, and propose a number of open problems related to these notions.
\end{abstract}

\noindent\textbf{Keywords:} graph class, graph parameter, clique number, independence number, Ramsey's theorem

\noindent\textbf{2020 Mathematics Subject Classification:} 05C69, 05C75, 05C65, 05D10, 05C85

\section{Introduction}

Graph parameters play a central role in structural and algorithmic graph theory, capturing key combinatorial properties of graphs.
Often, a desirable algorithmic or structural property is established by identifying a suitable graph parameter and showing that the target behaviour follows from its boundedness. For example, bounded degeneracy ensures that every subgraph has a low-degree vertex, enabling compact graph representations~\cite{Spi03}; bounded treewidth yields efficient algorithms for problems expressible in monadic second-order logic that are \textsf{NP}-hard in general~\cite{Cou90}.

A potential limitation of this approach is that boundedness of a graph parameter can substantially restrict graph structure: for instance, graphs of bounded treewidth are necessarily sparse, so boundedness of treewidth does not capture dense graphs. This motivates the study of relaxations or generalisations of boundedness of graph parameters that retain useful structural properties or algorithmic tractability while encompassing richer classes.   

The present paper is on the edge of two such lines of research. 
The first concerns \emph{clique-bounded parameters}, where the value of a parameter is bounded as a function of the clique number~$\omega$. A central example is the theory of $\chi$-bounded classes---hereditary classes in which the chromatic number $\chi(G)$ is bounded by a function of $\omega(G)$---which emerged as a generalisation of perfect graphs \cite{MR0951359}, and has been the subject of extensive study (see, e.g.,~\cite{MR4174126}).
A more recent example is the family of $(\tw,\omega)$-bounded classes, where the treewidth is bounded by a function of the clique number; such classes are known to have useful algorithmic applications (see~\cite{MR4334541,MR4332111}).

The second line focuses on \emph{independence variants} of parameters.
The independence variant of a parameter $\rho$, denoted $\alpha$-$\rho$, is obtained by relaxing the size constraint in the definition of $\rho$ to one on the independence number.
These variants are motivated in part by algorithmic considerations: for example, bounded $\alpha$-treewidth---a parameter introduced independently by Yolov~\cite{MR3775804} and by Dallard, Milani{\v c}, and {\v S}torgel (under the name \emph{tree-independence number}~\cite{dallard2022firstpaper,dallard2022secondpaper})---still guarantees polynomial-time solvability of the \textsc{Maximum Weight Independent Set} problem (and many others; see~\cite{DFGKM25,MR3775804,LMMORS24,dallard2022firstpaper,MR4640320}), while capturing significantly larger graph classes than bounded treewidth.
Another example of this type is the following independence variant of degeneracy: the \emph{$\alpha$-degeneracy} of a graph $G$ is the minimum $k$ such that every non-null induced subgraph of~$G$ has a vertex whose neighbourhood has independence number at most $k$.
This parameter was studied under the name \emph{inductive independence number} by Ye and Borodin~\cite{MR2916349}, from both structural and algorithmic points of view.

A simple Ramsey-type argument (see \cref{sec:Ramsey}) shows that bounded $\alpha$-$\rho$ always implies that $\rho$ is clique-bounded. 
This observation allows properties implied by clique-boundedness of $\rho$ to be transferred automatically to classes with bounded $\alpha$-$\rho$.  
It is natural to ask whether the converse holds: 

\begin{center}
Does clique-boundedness of $\rho$ always imply boundedness of $\alpha$-$\rho$?
\end{center}
Such an implication is not always straightforward.
For example, motivated by the complexity of the \textsc{Maximum Weight Independent Set} problem, it was conjectured that clique-boundedness of treewidth implies boundedness of $\alpha$-treewidth \cite{dallard2022secondpaper} (see also \cite{MR4334541}); this remained open for some time before being recently refuted by Chudnovsky and Trotignon~\cite{CT24}.
Nevertheless, the question of whether the \textsc{Maximum Weight Independent Set} problem admits a polynomial-time algorithm in any hereditary graph class with clique-bounded treewidth, raised by Dallard, Milani{\v c}, and {\v S}torgel in~\cite{MR4334541,dallard2022secondpaper}, remains open.

Motivated by the example of treewidth, in this paper, we formalise the notions of clique-boundedness and independence variants of graph parameters, and investigate when the two notions coincide.
Specifically, a graph parameter~$\rho$ is said to be \emph{awesome} if clique-boundedness of $\rho$ is equivalent to boundedness of $\alpha$-$\rho$, and is \emph{awful} otherwise.
Formal treatment of these notions is given in \cref{sec:framework}.
In \cref{sec:td-pw}, we show that treedepth and pathwidth are awful parameters. 
In \cref{sec:awesome}, we identify some parameters that are awesome, and derive structural and algorithmic consequences of awesomeness.
These include, in particular:
\begin{itemize}
    \item an infinite family of graph parameters $\rho$ such that the \textsc{Maximum Weight Independent Set} problem admits a polynomial-time algorithm in any hereditary graph class with clique-bounded $\rho$ (\Cref{thm:MWIS,thm:MWIS-oct}); one of these results provides a partial answer to the aforementioned question by Dallard et al.~\cite{dallard2022secondpaper},
    \item  a characterisation of hereditary graph classes with bounded $\alpha$-vertex cover number (\Cref{nK2-Knn}), and 
    \item a dense analogue of a result of Bonamy et al.~\cite{MR4723425} establishing logarithmic treewidth in classes of graphs excluding some fixed biclique as a subgraph and without many pairwise independent cycles (\Cref{thm:OK-free}); this provides an affirmative answer to a question of Hilaire, Milanič, and Vasić~\cite{HMV25}.
\end{itemize}
We conclude with a discussion of further questions and open problems in \cref{sec:conclusion}. 
In the next section, we introduce basic notation and main notions.

\section{Preliminaries}\label{sec:prelim}

We denote by $\mathbb{N}$ the set of all positive integers.
All graphs considered in this paper are finite, simple, and undirected.
Let $G=(V,E)$ be a graph. 
We denote by $V(G)$ and $E(G)$ the vertex set and the edge set of~$G$, respectively. For a set $S \subseteq V(G)$, we denote by $G[S]$ the subgraph of~$G$ induced by $S$.
The \emph{neighbourhood} of a vertex $v$ in $G$, denoted by $N_G(v)$, is the set of vertices adjacent to $v$ in $G$. 
For a set $S \subseteq V(G)$, we denote $N_G(S) = (\bigcup_{v \in S} N_G(v)) \setminus S$. 
When the graph $G$ is clear from the context, we may omit the subscript and simply write $N(v)$ and $N(S)$.
A vertex $v$ is \emph{universal} in $G$ if $N_G(v) = V(G) \setminus \{v\}$. 
The \emph{degree} of $v$ is the cardinality of the set $N_G(v)$. 
The \emph{maximum degree} of a graph $G$, denoted by $\Delta(G)$, is the largest degree of a vertex in $G$. 
We use the standard notation $K_{p,q}$, $K_s$, $P_s$, and $C_s$, for the complete bipartite graph with $p$ and $q$ vertices in the parts, the complete graph, the path, and the cycle on $s$ vertices, respectively. 
A \emph{star} with $q$ leaves is the complete bipartite graph $K_{1,q}$.
Given a graph $H$ and a nonnegative integer $r$, we denote by $rH$ the vertex-disjoint union of $r$ copies of $H$.
An \emph{isomorphism} from a graph $G$ to a graph~$G'$ is a bijective function $f:V(G)\to V(G')$ such that two distinct vertices $u,v\in V(G)$ are adjacent in $G$ if and only if their images $f(u)$ and $f(v)$ are adjacent in~$G'$.
The graphs $G$ and~$G'$ are said to be \emph{isomorphic} (to each other) if there exists an isomorphism from $G$ to~$G'$.

A \emph{clique} is a set of pairwise adjacent vertices; an \emph{independent set} is a set of pairwise nonadjacent vertices.
The \emph{clique number} of~$G$, denoted by $\omega(G)$, is the maximum cardinality of a clique in $G$;
the \emph{independence number} of~$G$, denoted by $\alpha(G)$, is the maximum cardinality of an independent set in $G$.
For a set $S \subseteq V(G)$, the \emph{independence number of $S$} is the independence number $\alpha(G[S])$ of the induced subgraph $G[S]$.
A set $M\subseteq E(G)$ is a \emph{matching} in $G$ if no two edges in $M$ share a common vertex. 
A set $S\subseteq V(G)$ is a \emph{vertex cover} in $G$ if $S$ contains at least one vertex of every edge of~$G$. 
The \emph{vertex cover number} of~$G$, denoted by $\vc(G)$, is the minimum cardinality of a vertex cover in  $G$.
A set $S\subseteq V(G)$ is a \emph{feedback vertex set} in $G$ if $S$ contains at least one vertex of every cycle of~$G$. 
The \emph{feedback vertex set number} of~$G$, denoted by $\fvs(G)$, is the minimum cardinality of a feedback vertex set in $G$.
A set $S\subseteq V(G)$ is an \emph{odd cycle transversal} in $G$ if $S$ contains at least one vertex of every odd cycle of~$G$. 
The \emph{odd cycle transversal number} of~$G$, denoted by $\octt(G)$, is the minimum cardinality of an odd cycle transversal in $G$.

A \emph{tree decomposition} of a graph $G$ is a pair $\mathcal{T} =(T,\beta)$ where $T$ is a tree and $\beta$ is a mapping assigning to each node $x$ of the tree $T$ a subset $\beta(x)$ of the vertex set of~$G$ (also called a \emph{bag} of~$\mathcal{T}$) such that the following conditions hold: each vertex of~$G$ is contained in some bag, for each edge $e$ of~$G$ there exists a bag containing both endpoints of $e$, and for each vertex $v$ of~$G$, the set of nodes $x$ of $T$ such that $v\in \beta(x)$ induces a connected subgraph of $T$.
The \emph{width} of a tree decomposition $(T,\beta)$ is defined as the maximum value of $|\beta(x)|$ over all nodes $x$ of the tree $T$.\footnote{We deviate from the standard way of defining the width by subtracting $1$ from its value. This way, the width simply corresponds to the cardinality of the largest bag. 
This small difference, which does not affect boundedness, simplifies the generalisation to hyperparameters (see p.~\pageref{treewidth-hyperparameterisation}).}
The \emph{treewidth} of a graph $G$, denoted $\tw(G)$, is defined as the minimum width of a tree decomposition of~$G$.
A \emph{path decomposition} of a graph $G$ is a tree decomposition $(P,\beta)$ of~$G$ such that $P$ is a path.
The \emph{pathwidth} of a graph $G$, denoted $\pw(G)$, is defined as the minimum width of a path decomposition of~$G$.
A \emph{rooted tree} is a pair $(T,r)$ such that $T$ is a tree and $r$ is a vertex of $T$ called the \emph{root}.
A \emph{rooted forest} is a disjoint union of rooted trees; in particular, a rooted forest may have more than one root.
A \emph{leaf} of a rooted forest $F$ is a vertex of out-degree $0$ in the digraph obtained from $F$ by orienting the edges away from the roots.
Given a rooted forest $F$, the \emph{depth} of $F$ is the maximum number of vertices on a root-to-leaf path.
The \emph{transitive closure} of a rooted forest $F$ is the graph $G$ obtained from $F$ by adding edges to it so that each root-to-leaf path in $F$ becomes a clique in $G$ (and adding no other edges).
A \emph{treedepth decomposition} of a graph $G$ is a rooted forest $F$ such that $G$ is a subgraph of the transitive closure of $F$.
The \emph{treedepth} of a graph $G$, denoted $\td(G)$, is defined as the minimum depth of a treedepth decomposition of~$G$.

\section{The framework}\label{sec:framework}

In this section, we formalise the key concepts used in the paper. In \cref{sec:gp-gc-cp}, we recall the formal definitions of graph classes and graph parameters, and discuss how the boundedness of graph parameters can be used to define families of classes (i.e., class properties). We then consider more general ways of defining class properties, where one parameter is bounded as a function of another parameter, rather than by a constant. Prominent examples include $\chi$-bounded and $(\tw, \omega)$-bounded classes.  
To treat these approaches to defining class properties uniformly, we introduce in \cref{sec:hyperparameters} the notion of a \emph{hyperparameter}. 
In \cref{sec:hyperparameterisations}, we demonstrate how various classical graph parameters, as well as their independence variants (e.g., the tree-independence number), can be expressed as hyperparameters.  
In \cref{sec:Ramsey}, we observe that for any cardinality based hyperparameter $\rho$, boundedness of its independence variant in a class of graphs implies that $\rho$ is clique-bounded in that class. 
Finally, in \cref{sec:antiRamsey}, we formally state the central question of our study: for which hyperparameters does the converse implication hold?

\subsection{Graph parameters, graph classes, and class properties}\label{sec:gp-gc-cp}

\medskip
\noindent{\bf Graph classes and graph parameters.} 
A \emph{graph class} is a collection of graphs closed under isomorphism.
A graph class is \emph{hereditary} if it is closed under vertex deletion; in other words, for every graph in the class, all of its induced subgraphs also belong to the class. Given a graph class~$\mathcal{G}$, the \emph{hereditary closure} of~$\mathcal{G}$ is the class of all induced subgraphs of graphs in~$\mathcal{G}$.
For a set $\mathcal{F}$ of graphs, we say that a graph $G$ is \emph{$\mathcal{F}$-free} if no graph in $\mathcal{F}$ is isomorphic to an induced subgraph of~$G$. 
If $\mathcal{F}$ consists of a single graph $F$, we write $F$-free instead of $\{F\}$-free.  
It is a folklore fact that a graph class~$\mathcal{G}$ is hereditary if and only if there exists a set $\mathcal{F}$ of graphs such that a graph $G$ belongs to~$\mathcal{G}$ if and only if $G$ is $\mathcal{F}$-free.

A \emph{graph parameter} (also called \emph{graph invariant}) is a function mapping graphs to the set of integers\footnote{Or, more generally, to the set of real numbers; however, all parameters discussed in this paper will be integer-valued.} that does not distinguish between isomorphic graphs.
A graph parameter is \emph{monotone under induced subgraphs} if, for every graph $G$ and every induced subgraph $H$ of~$G$, the value of the parameter on $H$ is at most its value on $G$.
Many questions in graph theory can be formulated in terms of hereditary graph classes and graph parameters.
In particular, structural graph theory aims at understanding which substructures, when present in a graph, are ``responsible'' for a certain parameter to be large.
For example:
\begin{itemize}
\item Ramsey's theorem (see~\cite{MR1576401}) states that a necessary and sufficient condition for the number of vertices of a graph to be large is the presence of either a large clique or a large independent set.
\item The Grid-Minor Theorem of Robertson and Seymour (see~\cite{MR0854606}) states that a necessary and sufficient condition for the treewidth of a graph to be large is the presence of a large grid as a minor.
\item The Strong Perfect Graph Theorem due to Chudnovsky, Robertson, Seymour, and Thomas (see~\cite{MR2233847}) can be stated as follows.
A necessary and sufficient condition for the maximum difference, over all induced subgraphs, between the chromatic number and the clique number to be strictly positive is the presence of an induced odd cycle of length at least five in the graph or its complement.

\item More generally, understanding necessary and sufficient conditions under which large cliques are the only structures responsible for large chromatic number is the main research goal of the theory of $\chi$-boundedness (see~\cite{MR4174126} for a survey).
\end{itemize}

\medskip
\noindent{\bf From graph parameters to class properties.} 
Given a graph parameter $\rho$, we say that a graph class~$\mathcal{G}$ has \emph{bounded $\rho$} if there exists an integer $c$ such that for every graph $G \in \mathcal{G}$, every induced subgraph~$G'$ of~$G$ satisfies $\rho(G')\le c$.
Note that a graph class~$\mathcal{G}$ has bounded $\rho$ if and only if its hereditary closure has bounded $\rho$.
Our definition differs slightly from the more common definition stating that a graph class~$\mathcal{G}$ has bounded $\rho$ if there exists a constant $c$ such that $\rho(G)\le c$ for all graphs $G$ in~$\mathcal{G}$.
However, the two definitions coincide for hereditary graph classes, as well as for any graph parameter $\rho$ that is monotone under induced subgraphs.

A \emph{class property} is a family of graph classes. 
Given a class property $\mathcal{P}$ and a graph class~$\mathcal{G}$, we say that~$\mathcal{G}$ \emph{satisfies} $\mathcal{P}$ if $\mathcal{G}\in \mathcal{P}$.
Any graph parameter $\rho$ naturally leads to a class property $\mathcal{B}_\rho$, consisting of precisely those graph classes~$\mathcal{G}$ that have bounded~$\rho$.
The identification of substructures ``responsible'' for a parameter $\rho$ to be large leads to a characterisation of the corresponding class property $\mathcal{B}_\rho$.
For example:
\begin{itemize}
\item The Grid-Minor Theorem states that a minor-closed graph class has bounded treewidth if and only if it excludes some grid.
\item Ramsey's theorem states that a hereditary graph class has bounded order if and only if it excludes some complete graph and some edgeless graph.
\end{itemize}

\medskip
\noindent{\bf Generalisations.} 
Further class properties can be obtained by considering two graph parameters at once, by requiring that one of the parameters is bounded by a function of the other one.
This includes the following class properties:
\begin{itemize}
    \item A hereditary graph class~$\mathcal{G}$ is \emph{$\chi$-bounded} if in the class~$\mathcal{G}$ large cliques are the only structures responsible for large chromatic number, that is, there exists a function $f$ such that every graph $G$ in~$\mathcal{G}$ satisfies $\chi(G)\le f(\omega(G))$, where $\chi(G)$ and $\omega(G)$ denote the chromatic number and the clique number of~$G$, respectively (see~\cite{MR0951359}).
\item A hereditary graph class~$\mathcal{G}$ is \emph{$(\tw,\omega)$-bounded} if in the class~$\mathcal{G}$ large cliques are the only structures responsible for large treewidth (see, e.g.,~\cite{MR4334541,MR4332111}).
\end{itemize}
The above two examples are special cases of the following framework.
Given two graph parameters $\rho$ and $\sigma$, we say that a graph class~$\mathcal{G}$ is \emph{$(\rho,\sigma)$-bounded} (or \emph{has $\sigma$-bounded $\rho$}) if there exists a nondecreasing function $f$ such that, for every graph $G$ in the class~$\mathcal{G}$, every induced subgraph~$G'$ of~$G$ satisfies $\rho(G')\le f(\sigma(G'))$.

In general, there is no assumption on~$\mathcal{G}$ to be hereditary, and this is the reason why the condition is imposed on all induced subgraphs of graphs in~$\mathcal{G}$. 
However, in our paper we focus on hereditary graph classes, in which case it of course suffices to require that $\rho(G)\le f(\sigma(G))$ for all graphs $G$ in~$\mathcal{G}$.

In the above terminology, $\chi$-bounded and $(\tw,\omega)$-bounded graph classes are said to have \emph{$\omega$-bounded chromatic number} and \emph{$\omega$-bounded treewidth}, respectively.
The latter class property is a subfamily of the class property of $\omega$-bounded expansion studied by Jiang, Ne\v{s}et\v{r}il, and Ossona de Mendez~\cite{MR4513822}.
Several other class properties captured by the above framework were studied in the literature, including $\chi$-bounded colouring number (see the book by Jensen and Toft~\cite{MR1304254}), $\chi$-bounded minimum degree (see~\cite{MR2811077}), degeneracy-bounded treewidth (by Gy\'{a}rf\'{a}s and Zaker~\cite{MR3794363}), and $\omega$-bounded pathwidth (by Hermelin, Mestre, and Rawitz~\cite{MR3215457}).

\medskip
\noindent{\bf Clique-bounded class properties.} 
As exemplified by the properties of $\chi$-boundedness, $\omega$-bounded treewidth, and $\omega$-bounded pathwidth, an important special case of the above framework for defining properties of graph classes is the case when the parameter $\sigma$ is the clique number, $\omega$.
In this case, given a graph parameter $\rho$, we may also say that a graph class that has $\omega$-bounded $\rho$ has \emph{clique-bounded $\rho$} and denote by $\mathcal{C}_\rho$ the corresponding class property. 
Thus, we speak of graph classes with clique-bounded pathwidth, clique-bounded treewidth, etc.

Note that for any graph parameter $\rho$, a graph class~$\mathcal{G}$ belongs to the class property $\mathcal{C}_\rho$ if and only if there exists a function $f$, called a \emph{binding function}, such that for every graph $G$ in the class~$\mathcal{G}$, every induced subgraph~$G'$ of~$G$ satisfies $\rho(G')\le f(\omega(G'))$.
Furthermore, any hereditary graph class~$\mathcal{G}$ satisfying the class property $\mathcal{B}_\rho$ (bounded $\rho$) also satisfies the property $\mathcal{C}_\rho$ (clique-bounded $\rho$), which can be observed by considering the constant function $f(p) = c$, for all $p\in \mathbb{N}$, where $c$ is any integer such that $\rho(G)\le c$ for all graphs $G$ in~$\mathcal{G}$.

\subsection{Hyperparameters}\label{sec:hyperparameters}

An \emph{annotated graph} is a tuple $(G,X)$ such that $G$ is a graph and $X\subseteq V(G)$.
An \emph{annotated graph parameter} is a function $\lambda$ that assigns to every annotated graph $(G,X)$ an integer such that for any graph~$G'$ isomorphic to $G$ and any isomorphism $f$ from $G$ to~$G'$, it holds that $\lambda(G,X) = \lambda(G',f(X))$, where $f(X) = \{f(x)\colon x\in X\}$.
Two annotated graph parameters that play an important role in the rest of the paper are: (i) the cardinality function, defined by $\mathtt{card}(G,X) \coloneqq |X|$ for all annotated graphs $(G,X)$, and (ii)~the independence number of the subgraph of~$G$ induced by $X$, defined by $\alpha(G,X) \coloneqq \alpha(G[X])$ for all annotated graphs $(G,X)$.

Given a graph $G$, a \emph{hypergraph over $G$} is a nonempty set of subsets of $V(G)$.
Given a finite nonempty family $\mathcal{F}$ of hypergraphs over a graph $G$ and an annotated graph parameter $\lambda$, we define the \emph{$\lambda$-min-max value} of $(G,\mathcal{F})$ to be the quantity \[\lambda\text{-}\minmax(G,\mathcal{F})= \min_{H\in \mathcal{F}}\max_{X\in H}\lambda(G,X)\,.\]
The \emph{$\lambda$-max-min value} of $(G,\mathcal{F})$, defined analogously, is
\[\lambda\text{-}\maxmin(G,\mathcal{F})= \max_{H\in \mathcal{F}}\min_{X\in H}\lambda(G,X)\,.\]
A \emph{hypermapping} is a mapping $\mathcal{F}$ that associates to each graph $G$ a finite nonempty family $\mathcal{F}_G$ of hypergraphs over $G$ such that, for any graph~$G'$ isomorphic to $G$ and any isomorphism $f$ from $G$ to~$G'$, we have $\mathcal{F}_{G'} = \{\varphi(H)\colon H\in \mathcal{F}_G\}$ where $\varphi(H) = \{f(X)\colon X\in H\}$ for all $H\in \mathcal{F}_G$.
A \emph{graph hyperparameter}, or simply a \emph{hyperparameter}, is a tuple $\rho = (\lambda, \mathrm{opt}, \mathcal{F})$ such that $\lambda$ is an annotated graph parameter, \hbox{$\mathrm{opt}\in \{\mathrm{minmax},\mathrm{maxmin}\}$}, and $\mathcal{F}$ is a hypermapping.

As we show next, to every hyperparameter one can associate a graph parameter.

\begin{proposition}\label{hyperparameter-to-parameter}
Let $\rho = (\lambda, \mathrm{opt}, \mathcal{F})$ be a hyperparameter.
Given a graph $G$, we denote by $\rho(G)$ the value of $\lambda\text{-}\mathrm{opt}(G,\mathcal{F}_G)$.
Then, $\rho$ is a graph parameter, that is, if $G$ and~$G'$ are isomorphic graphs, then $\rho(G) = \rho(G')$. 
\end{proposition}

\begin{proof}
Consider first the case when $\mathrm{opt}= \mathrm{minmax}$.
Let $G$ and~$G'$ be isomorphic.
Fix an isomorphism $f:G\to G'$.
Note that, since $\lambda$ is an annotated graph parameter, for every set $X\subseteq V(G)$, we have $\lambda(G,X) = \lambda(G',f(X))$.
Recall also that $\mathcal{F}_{G'} = \{\varphi(H)\colon H\in  \mathcal{F}_G\}$ where $\varphi(H) = \{f(X)\colon X\in H\}$ for all $H\in \mathcal{F}_G$.
The definition of $\varphi$ implies that for every hypergraph $H\in \mathcal{F}_G$, we have that \[\max_{X\in H}\lambda(G,X) = \max_{X\in H}\lambda(G',f(X)) = \max_{X'\in \varphi(H)}\lambda(G',X')\,.\]
Consequently,
\begin{flalign*}
    \rho(G) &= \lambda\text{-}\minmax(G,\mathcal{F}_G) = \min_{H\in \mathcal{F}_G}\max_{X\in H}\lambda(G,X) = \min_{H\in \mathcal{F}_G}\max_{X'\in \varphi(H)}\lambda(G',X')\\
&= \min_{H'\in \mathcal{F}_{G'}}\max_{X'\in H'}\lambda(G',X') = \lambda\text{-}\minmax(G',\mathcal{F}_{G'}) = \rho(G')\,,
\end{flalign*}
where the fourth equality follows from the fact that $\mathcal{F}_{G'} = \{\varphi(H)\colon H\in  \mathcal{F}_G\}$.

The case when $\mathrm{opt}= \mathrm{maxmin}$ is similar.
\end{proof}

We say that a hyperparameter $\rho = (\lambda,\mathrm{opt}, \mathcal{F})$ is a \emph{hyperparameterisation} of the corresponding graph parameter $\rho(G)=\lambda\text{-}\mathrm{opt}(G,\mathcal{F}_G)$, as given by \cref{hyperparameter-to-parameter}.
Let us observe that every graph parameter $\rho$ mapping graphs to the set of integers admits multiple hyperparameterisations, as can be seen by considering any tuple $(\lambda, \opt, \mathcal{F})$ such that $\lambda$ is the annotated graph parameter that maps every annotated graph $(G,X)$ to $\rho(G)$, $\opt\in \{\minmax,\maxmin\}$, and $\mathcal{F}$ is \textsl{any} hypermapping.
However, such hyperparameterisations are artificial and, hence, uninteresting.
We are interested in hyperparameterisations whose definition reflects the definition of a certain graph parameter $\rho$ in a more structural way, in the sense that the hypermapping corresponds to the set of feasible solutions to the graph optimisation problem defining $\rho$.
Several examples are given in \cref {sec:hyperparameterisations}.

In this paper, we focus on two particular kinds of hyperparameterisations.
We say that a hyperparameter $(\lambda, \opt, \mathcal{F})$ is:
\begin{itemize}
    \item \emph{cardinality based} (or simply \emph{basic}) if $\lambda$ is the cardinality function, that is, for all annotated graphs $(G,X)$, we have $\lambda(G,X) = \mathtt{card}(G,X) \coloneqq |X|$,
    
    \item \emph{independence based} if $\lambda$ is the independence number of the subgraph of~$G$ induced by $X$, that is, for all annotated graphs $(G,X)$, we have $\lambda(G,X) = \alpha(G,X) \coloneqq \alpha(G[X])$.
\end{itemize}

Given a hyperparameter $\rho = (\lambda,\opt, \mathcal{F})$ and an annotated graph parameter $\eta$, we define the \emph{$\eta$-variant} of $\rho$, denoted by $\eta\text{-}\rho$, as the hyperparameter $(\eta, \opt, \mathcal{F})$.
Throughout the paper, we mostly consider basic hyperparameters and their $\alpha$-variants.
Given a basic hyperparameter $\rho$, we also refer to its $\alpha$-variant as the \emph{independence variant} of $\rho$.

\medskip
\cref{hyperparameter-to-parameter} allows us to interpret every hyperparameter $\rho$ as some graph parameter, and, hence, naturally leads to the notions of graph classes that have \emph{bounded $\rho$}, \emph{bounded $\alpha\text{-}\rho$}, or have \emph{clique-bounded~$\rho$}.

\subsection{Canonical hyperparameterisations and variants}\label{sec:hyperparameterisations}

Many classical graph parameters can be expressed as basic hyperparameters.
In particular, we introduce the following hyperparameterisations of several well-known graph parameters.
For simplicity, we refer to these hyperparameterisations as \emph{canonical} hyperparameterisations of the corresponding parameters.
In all these examples, the corresponding hyperparameter $\rho = (\lambda, \opt, \mathcal{F})$ is basic, that is, $\lambda = \mathtt{card}$.

\medskip
\noindent For \textbf{order}, $n(G) = |V(G)|$, we can take a trivial hyperparameterisation: $\opt\in \{\minmax,\maxmin\}$ and $\mathcal{F}_G = \{\{V(G)\}\}$.

\medskip
\noindent For \textbf{maximum degree}, we take $\opt = \minmax$ and $\mathcal{F}_G = \big\{\{N_G(v) \colon v\in V(G)\}\big\}$.

\medskip
\noindent For \textbf{independence number}, we take\footnote{One could also take $\opt = \maxmin$ and $\mathcal{F}_G = \Big\{\big\{\{ X\}\big\} \colon X \text{ is an independent set in }G\Big\}$.}
    \[
    \opt = \minmax \text{ and } \mathcal{F}_G = \big\{\{X \colon X \text{ is an independent set in $G$}\}\big\}.
    \]
    That is, with $\opt$ and $\mathcal{F}_G$ as above, all graphs $G$ satisfy $\alpha(G) = \mathtt{card}\text{-}\opt(G,\mathcal{F}_G)$.
    
\medskip
\noindent For \textbf{vertex cover number}, we take\footnote{One could also take $\opt = \minmax$ and $\mathcal{F}_G = \Big\{\big\{\{ X\}\big\} \colon X \text{ is a vertex cover in }G\Big\}$.}
    \[
    \opt = \maxmin \text{ and } \mathcal{F}_G = \big\{\{X \colon X \text{ is a vertex cover in $G$}\}\big\}.
    \]
\medskip
\noindent For \textbf{feedback vertex number}, we take \[
    \opt = \maxmin \text{ and } \mathcal{F}_G = \big\{\{X \colon X \text{ is a feedback vertex set in $G$}\}\big\}\,.\]

\begin{remark}\label{hyperpar-optimisation}
As indicated by these examples, whenever $\rho$ is a graph minimisation parameter defined as the smallest cardinality of a vertex subset of~$G$ satisfying a certain property $P$, one can define the canonical hyperparameterisation of $\rho$ as follows:
\[
    \opt = \maxmin \text{ and } \mathcal{F}_G = \big\{\{X\subseteq V(G) \colon X \text{ satisfies $P$ in $G$}\}\big\}\,,\]
and similarly for maximisation parameters, replacing $\maxmin$ with $\minmax$.
\end{remark}

\medskip
\noindent For \textbf{degeneracy}, we take $\opt = \maxmin$ and $\mathcal{F}_G = \{\{ N_{G'}(v) \colon v \in V(G')\}\colon  G'$ is a non-null induced subgraph of $G\}$.

\medskip
\noindent For \textbf{treewidth},\label{treewidth-hyperparameterisation}
 we take $\opt = \minmax$ and $\mathcal{F}_G$ to be the family of all tree decompositions of~$G$ considered as hypergraphs whose hyperedges are the bags of the tree decomposition.

\medskip
\noindent For \textbf{pathwidth}, we take, similarly as for treewidth, $\opt = \minmax$ and $\mathcal{F}_G$ to be the family of all path decompositions of~$G$.

\medskip
\noindent For \textbf{treedepth}, we take $\opt = \minmax$ and $\mathcal{F}_G$ to be the set of all treedepth decompositions of~$G$ considered as hypergraphs whose hyperedges correspond to the vertex sets of each root-to-leaf path in the decomposition.

\medskip
\noindent 
For \textbf{chromatic number}, we take $\opt = \minmax$ and $\mathcal{F}_G$ to be the family of hypergraphs, one for each proper colouring of~$G$, whose hyperedges are all the sets of vertices of~$G$ with pairwise different colours.

\medskip
In the rest of the paper, unless specified otherwise, we identify the above graph parameters with their canonical hyperparameterisations.

Recall that each graph parameter that admits a hyperparameterisation $\rho = (\mathtt{card},\opt, \mathcal{F})$ also admits its independence variant, also called the \emph{$\alpha$-variant} of $\rho$.
For instance:
\begin{itemize}
    \item If $\rho$ is either the order or the independence number, then its $\alpha$-variant is the independence number.
    
    \item If $\rho$ is the maximum degree, then its $\alpha$-variant is the maximum number of leaves of an induced star in a graph $G$.
    This parameter, recently studied under the name \emph{local independence number} by Gollin, Hatzel, and Wiederrecht~\cite{gollin2025graphssimplestructuremaximal}, is a lower bound on the Cartesian dimension of a graph (see~\cite{DBLP:conf/mfcs/MilanicMM17}).

    \item If $\rho$ is the degeneracy, then its $\alpha$-variant is the \emph{inductive independence number}, studied by Ye and Borodin~\cite{MR2916349}.

    \item If $\rho$ is the treewidth (plus one), then its $\alpha$-variant is the \emph{tree-independence number} (also called \emph{$\alpha$-treewidth}), introduced independently by Yolov~\cite{MR3775804} and by Dallard, Milani{\v c}, and {\v S}torgel~\cite{dallard2022firstpaper,dallard2022secondpaper}, and studied  in~\cite{DBLP:conf/soda/AhnGHK25,DFGKM25,LMMORS24,chudnovsky2025treeindependencenumbervithetas,chudnovsky2025treewidthcliqueboundednesspolylogarithmictreeindependence,CT24,gollin2025graphssimplestructuremaximal,abrishami2023tree,chudnovsky2025treeindependencenumberv,DKKMMSW2024,HMV25,chudnovsky2025treeindependencenumberivsoda,chudnovsky2024treeindependencenumberiii,MR4955546}
   (see also Adler~\cite{adler2006width} for a more general context).

   \item If $\rho$ is the pathwidth (plus one), then its $\alpha$-variant is the \emph{path-independence number}, considered recently by Beisegel et al.~\cite{beisegel_et_al:LIPIcs.SWAT.2024.7} in relation with the simultaneous interval number, a new width parameter that measures the similarity to interval graphs.
\end{itemize}

Although not studied in this paper, other variants of basic hyperparameters have been considered in the literature, especially in relation to treewidth and pathwidth. 
This includes:
\begin{itemize}
    \item the chromatic variants of treewidth and pathwidth, corresponding to the \emph{tree-chromatic number} and \emph{path-chromatic number}, respectively, introduced by Seymour~\cite{MR3425243} and studied in~\cite{huynhkim17,BARRERACRUZ2019206,Huynh2021};

    \item the clique cover variant of treewidth (see~\cite{doi:10.1137/20M1320870,abrishami2023tree}; see also~\cite{zbMATH07567466,zbMATH07365274} for a related variant where the annotated graph parameter $\lambda$ maps each annotated graph $(G,X)$ to the size of the smallest subset $Y \subseteq X$ such that $X\setminus Y$ is a clique in $G$);
    
\item the variant of treewidth defined by replacing the cardinality in the corresponding hyperparameter with the annotated graph parameter $\lambda$ that maps each annotated graph $(G,X)$ to the maximum cardinality of an induced matching in $G$ every edge of which has an endpoint in $X$; this hyperparameter corresponds to the \emph{induced matching treewidth}, introduced  by Yolov~\cite{MR3775804} (under the name \emph{minor-matching hypertreewidth}) and studied by Lima et al.~\cite{LMMORS24}, Abrishami et al.~\cite{abrishami2024excludingcliquebicliquegraphs}, and Bodlaender et al.~\cite{bodlaender2025findingsparseinducedsubgraphs};

\item further variants of treewidth were studied in the context of hypergraphs by Marx~\cite{MR3144912}.
\end{itemize}

Let us mention an example of how some known inequalities between graph parameters can be easily lifted to analogous inequalities between hyperparameters.
Treewidth, pathwidth, and treedepth of a graph $G$ are related by the following chain of inequalities: $\tw(G)\le \pw(G)\le \td(G)$. (The former inequality follows directly from the definitions; for the latter one, see, e.g.,~\cite{MR1317666,MR2920058}.)
Although the proof below does not involve any new ideas, we include it for the sake of completeness and since it is adapted to the notations of this paper.

\begin{proposition}\label{prop:eta}
For every annotated graph parameter $\lambda$ and every graph $G$, we have
\[\lambda\text{-}\tw(G)\le \lambda\text{-}\pw(G)\le \lambda\text{-}\td(G)\,.\]
\end{proposition}

\begin{proof}
Let $G$ be a graph.
By the definition of $\lambda\text{-}\tw$, we have 
\[\lambda\text{-}\tw(G) = \lambda\text{-}\minmax(G,\mathcal{F}_G^{\tw})= \min_{H\in \mathcal{F}_G^{\tw}}\max_{X\in H}\lambda(G,X)\,,\]
where $\mathcal{F}_G^{\tw}$ is the family of all hypergraphs over $G$ whose hyperedges are the bags of some tree decomposition of~$G$.
Similarly, 
\[\lambda\text{-}\pw(G) = \lambda\text{-}\minmax(G,\mathcal{F}_G^{\pw})\quad \text{ and }\quad \lambda\text{-}\td(G) = \lambda\text{-}\minmax(G,\mathcal{F}_G^{\td})\,,\] 
where $\mathcal{F}_G^{\pw}$ is the family of all hypergraphs over $G$ whose hyperedges are the bags of some path decomposition of~$G$, and $\mathcal{F}_G^{\td}$ is the family of all hypergraphs over $G$ whose hyperedges are the vertex sets of root-to-leaf paths of some treedepth decomposition of~$G$.
To prove $\lambda\text{-}\tw(G)\le \lambda\text{-}\pw(G)\le \lambda\text{-}\td(G)$, it suffices to show that $\mathcal{F}_G^{\td}\subseteq\mathcal{F}_G^{\pw}\subseteq \mathcal{F}_G^{\tw}$.
The fact that every path decomposition is also a tree decomposition implies the second inclusion.
To see the first one, consider a treedepth decomposition $F$ of~$G$.
Fix an ordering of the leaves of the decomposition that can be obtained from a DFS traversal on $F$, starting the traversal of each component from its root.
The corresponding ordering of the root-to-leaf paths in $F$ yields a path decomposition of~$G$ in which the bags are exactly the vertex sets of the root-to-leaf paths.
This implies that $\mathcal{F}_G^{\td}\subseteq\mathcal{F}_G^{\pw}$.
\end{proof}

Under some mild technical conditions on $\lambda$, the chain of inequalities from \Cref{prop:eta} can also be extended to include the $\lambda$-variant of vertex cover number; see \cref{prop:lambda-td-vc}.

\subsection{Clique-boundedness via Ramsey's theorem}\label{sec:Ramsey}

We now prove that for basic hyperparameters, boundedness of the independence variant implies clique-boundedness.
This result follows from Ramsey's theorem~\cite{MR1576401}, which states that for any two nonnegative integers $a$ and $b$, there exists a least nonnegative integer $R(a,b)$ such that any graph with at least $R(a,b)$ vertices has either a clique of size $a$ or an independent set of size $b$.

\begin{proposition}\label{prop:main}
Let $\rho$ be a basic hyperparameter and let~$\mathcal{G}$ be a graph class that has bounded $\alpha\text{-}\rho$.
Then~$\mathcal{G}$ has clique-bounded $\rho$ with a polynomial binding function.
\end{proposition}

\begin{proof}
Let $\rho = (\mathtt{card}, \opt, \mathcal{F})$ and~$\mathcal{G}$ be as in the proposition.
Then, there exists an integer $k$ such that for every graph $G$ in~$\mathcal{G}$, every induced subgraph~$G'$ of~$G$ satisfies $\alpha\text{-}\rho(G')\le k$. 

In order to show that~$\mathcal{G}$ has clique-bounded $\rho$, we show that for every graph $G$ in the class~$\mathcal{G}$, every induced subgraph~$G'$ of~$G$ satisfies $\rho(G')\le f(\omega(G'))$, where $f:\mathbb{N}\to \mathbb{N}$ is the function defined by
\[f(p) = R(p+1,k+1)-1\,,\]
for all $p\in \mathbb{N}$.
The standard proof of Ramsey’s theorem is based on the inequality $R(k,\ell) \leq R(k-1, \ell)+R(k,\ell-1)$, for all $k$, $\ell \geq 2$, which implies that $R(k,\ell) \leq \binom{k+\ell-2}{k-1}$, for all positive integers $k$ and $\ell$.
For fixed $k$, this binomial coefficient is a polynomial in $\ell$.\footnote{For a different bound, also showing that the Ramsey number $R(k,\ell)$ with one argument fixed is a polynomial function of the other argument, see, e.g.,~\cite[Theorem 1.9]{MR4672178}.}

Let $G$ be a graph in~$\mathcal{G}$ and let~$G'$ be an induced subgraph of~$G$.
Let $p$ be the clique number of~$G'$.
Observe that, for any subset $S$ of $V(G')$ such that $\alpha(G'[S])\le k$, we have
\[|S|\le R(p+1, k+1)-1 = f(p)\,,\]
where the inequality follows from Ramsey's theorem, using the fact that $\omega(G'[S]) \le \omega(G') = p$.

Consider first the case when $\mathrm{opt}= \mathrm{minmax}$.
Then
\[\min_{H \in \mathcal{F}_{G'}} \max_{X \in H} \alpha(G'[X]) = \alpha\text{-}\rho(G') \le k\,,\]
that is, there exists a hypergraph $H^* \in \mathcal{F}_{G'}$ such that $\alpha(G'[X])\le k$ for all $X\in H^*$.
Hence, $|X|\le f(p)$ for all $X\in H^*$, and consequently,
\[\rho(G') = \min_{H \in \mathcal{F}_{G'}} \max_{X \in H}|X|\le  \max_{X \in H^*}|X| \le f(p) = f(\omega(G'))\,.\]

Suppose now that $\mathrm{opt}= \mathrm{maxmin}$.
Then
\[\max_{H \in \mathcal{F}_{G'}} \min_{X \in H} \alpha(G'[X]) = \alpha\text{-}\rho(G') \le k\,,\]
that is, for every hypergraph $H \in \mathcal{F}_{G'}$ there exists some $X_H\in H$ such that $\alpha(G'[X_H])\le k$, implying that $\min_{X \in H}|X| \le |X_H|\le f(p)$.
Consequently,
\[\rho(G') = \max_{H \in \mathcal{F}_{G'}} \min_{X \in H}|X| \le f(p) = f(\omega(G'))\,.\qedhere\]
\end{proof}

\subsection{Inverting Ramsey}\label{sec:antiRamsey}

Recall that, for a hyperparameter $\rho$, we denote by $\mathcal{B}_\rho$ and $\mathcal{C}_\rho$ the class properties of having bounded~$\rho$ and clique-bounded $\rho$, respectively.
Furthermore, for the case when $\rho$ is a basic hyperparameter, \cref{prop:main} applies, which can be expressed in terms of the inclusion 
\begin{equation}
\mathcal{B}_{\alpha\textrm{-}\rho}\subseteq \mathcal{C}_{\rho}\label{eq:inclusion}
\end{equation}
between the corresponding class properties.
This motivates the following definition.
A basic hyperparameter $\rho$ is said to be \emph{awesome} (or, better yet, \emph{\aw}) 
if \[\mathcal{B}_{\alpha\textrm{-}\rho}= \mathcal{C}_{\rho}\,,\]
that is, clique-boundedness of $\rho$ is equivalent to boundedness of its independence variant.
We say that $\rho$ is \emph{awful} if it is not awesome.
By \cref{eq:inclusion}, a basic hyperparameter is awesome if and only if clique-boundedness of $\rho$ implies boundedness of its independence variant, that is, $\mathcal{C}_{\rho}\subseteq \mathcal{B}_{\alpha\textrm{-}\rho}$.

While these definitions can be extended to graph parameters listed in \cref{sec:hyperparameterisations} by applying them to their canonical hyperparameterisations, in general one needs to be careful when applying them to graph parameters: if the choice of a hyperparameterisation is not obvious, different hyperparameterisations of the same graph parameter may have different $\alpha$-variants.
For example, in \cref{sec:hyperparameterisations} we presented one particular hyperparameterisation of the chromatic number, but undoubtedly many others are possible.

\begin{question}\label{main-question}
Which basic hyperparameters are awesome?
\end{question}

It turns out that this question is not always trivial.\label{discussion-awful-chi}
As shown by \cref{prop:main}, a necessary condition for a hyperparameter $\rho$ to be awesome is that for any hereditary graph class, clique-boundedness of $\rho$ is equivalent to polynomial clique-boundedness.
For the chromatic number, this was conjectured by Esperet~\cite{esperet2017graph}, and disproved only recently by Briański, Davies, and Walczak~\cite{MR4707561}.
Thus, by \cref{prop:main}, chromatic number is awful.
Note that this argument rules out awesomeness not only of the canonical hyperparameterisation of the chromatic number given in \cref{sec:hyperparameterisations}, but of \textsl{any} hyperparameterisation of the chromatic number.
Furthermore, motivated by algorithmic considerations, Dallard, Milani\v{c}, and \v{S}torgel conjectured in \cite{dallard2022secondpaper} that treewidth is awesome.
This conjecture was recently disproved by Chudnovsky and Trotignon (see~\cite{CT24}); indeed, they constructed both graph classes with clique-bounded treewidth that do not have polynomially clique-bounded treewidth, as well as graph classes with polynomially clique-bounded treewidth that have unbounded $\alpha$-treewidth.

It is a natural question whether more restricted parameters such as pathwidth, treedepth, and vertex cover number are awesome.
For pathwidth, the fact that for any hereditary graph class, clique-boundedness of pathwidth is equivalent to polynomial clique-boundedness was shown recently by Hajebi~\cite{hajebi2025polynomialboundspathwidth}; however, this result does not say anything about whether pathwidth is awesome.
Indeed, we show in this paper that both pathwidth and treedepth are awful (\cref{sec:td-pw}).
On the other hand, we identify infinitely many awesome graph parameters (\cref{sec:awesome}), including vertex cover, feedback vertex numbers, and odd cycle transversal number (see \cref{fig:graph-parameters} for a partial illustration of our results).

\begin{figure}[ht]
    \centering
    \resizebox{0.95\linewidth}{!}{
    \begin{tikzpicture}[
          font=\small,
          >=Latex,
          param/.style={draw, rounded corners=2pt, inner sep=3pt, align=center, very thick},
          greenbox/.style={param, draw=green!80, fill=green!40},
          orangebox/.style={param, draw=orange!80, fill=orange!40},
          bluebox/.style={param, draw=blue},
          blackarrow/.style={-Latex, line width=0.5pt, draw=gray!70},
        ]
        \node[greenbox] (n) at (0,7.2) {order};
       
        \node[greenbox] (vc) at (0,6) {vertex cover number};
        
        \node[orangebox] (td)  at (0,4.2) {treedepth};
        
        \node[orangebox] (pw)  at (0,2.4) {pathwidth};
        
        \node[greenbox] (fvs) at (3.2,4.5) {feedback vertex number};
        \node[greenbox] (tw2) at (5,3.5) {$(\tw,3)$-modulator number};
        \node[greenbox] (tw3) at (6.8,2.5) {$(\tw,4)$-modulator number};
        
        \node[orangebox] (tw)  at (0,0.6) {treewidth};

        \node[greenbox] (oct) at (-4,1.5) {odd cycle transversal number};
        \node[greenbox] (chi3) at (-5.5,0.5) {$(\chi,3)$-modulator number};
        \node[greenbox] (chi4) at (-7,-0.5) {$(\chi,4)$-modulator number};

        \node[greenbox] (Delta) at (3.4,0.5) {maximum degree};
        
        \node[orangebox] (chi) at (0,-1.2) {chromatic number};
        \node[greenbox] (omega) at (0,-2.4) {clique number};
        
        \draw[blackarrow] (pw) -- (td);
        \draw[blackarrow] (tw.north) -- (pw);
        \draw[blackarrow] (td) -- (vc);
        \draw[blackarrow] (vc) -- (n);
        \draw[blackarrow] (fvs)-- (vc.south);
        \draw[blackarrow] (tw2)-- (fvs);
        \draw[blackarrow] (tw3)-- (tw2);
        \draw[blackarrow] (7.8,1.9)-- (tw3);
        \node[color=gray!50] at (8.1, 1.84) {{\Large $\ddots$}};
        
        \draw[blackarrow] (tw.north) -- (fvs);
        \draw[blackarrow] (tw.north) -- (tw2.west);
        \draw[blackarrow] (tw.north) -- (tw3.west);

        \draw[blackarrow] (oct.north)-- (vc.south);
        \draw[blackarrow] (chi3)-- ($(oct.south) - (0.6,0)$);
        \draw[blackarrow] (chi4)-- ($(chi3.south) - (0.6,0)$);

        \draw[blackarrow] (chi.north) -- (oct);
        \draw[blackarrow] (chi.north) -- (chi3);
        \draw[blackarrow] (chi.north) -- (chi4);
        \draw[blackarrow] (-7.7,-1.2) -- ($(chi4.south)-(0.2,0)$);
        \node[color=gray!50, rotate=-116] at (-7.9, -1.47) {{\Large $\ddots$}};
        
        \draw[blackarrow] (chi.north) -- (Delta.south);
        
        \draw[blackarrow] (chi)  -- (tw);
        \draw[blackarrow] (omega) -- (chi);
    \end{tikzpicture}
    }
    \caption{The diagram illustrates various \tcbox[nobeforeafter, tcbox raise base,
    colback=green!40, colframe=green!80, left=3pt, right=3pt, top=4pt, bottom=1pt]{awesome} and \tcbox[nobeforeafter, tcbox raise base,
    colback=orange!40, colframe=orange!80, left=3pt, right=3pt, top=1pt, bottom=1pt]{awful} graph parameters (with respect to their canonical hyperparameterisations, as defined in \Cref{sec:hyperparameterisations}), along with some of the relations between them. 
    An arrow from a parameter $\rho_1$ to a parameter $\rho_2$ indicates that boundedness of $\rho_2$ in a class of graphs implies boundedness of $\rho_1$ in that class. For clarity, not all such relations are shown in the figure.}
    \label{fig:graph-parameters}
\end{figure}

\section{Awful graph parameters}
\label{sec:td-pw}
We show that neither treedepth nor pathwidth are awesome graph parameters, by constructing a graph class that has clique-bounded treedepth but unbounded $\alpha$-pathwidth.
The crucial ingredient used for this result is a graph transformation that increases $\alpha$-pathwidth by exactly one.
We introduce this transformation in \cref{sec:alpha-pw}, and use it 
in \cref{sec:bddtd-not-alpha-bdd-pw} to construct a class witnessing awfulness of the two parameters.

\subsection{Increasing \texorpdfstring{$\alpha$}{α}-pathwidth by 1}\label{sec:alpha-pw}

The transformation is based on the following graph. The \emph{subdivided claw} or \emph{s-claw} is the graph obtained from subdividing once each edge of the claw.
Given a graph $G$, the \emph{s-claw substitution} of~$G$ is the graph denoted by $\s(G)$ and defined as the graph obtained by substituting a copy of~$G$ into each of the three leaves (vertices of degree $1$) of the subdivided claw.
More formally, $\s(G)$ is the graph obtained from the disjoint union of three copies of~$G$, say $G_1$, $G_2$ and $G_3$, by adding three new vertices $v_1$, $v_2$ and $v_3$, making each vertex $v_i$ adjacent to all vertices in $G_i$, and adding a fourth new vertex $w$ adjacent precisely to $v_1$, $v_2$ and~$v_3$.

\begin{theorem}\label{increase-path-alpha-by-one}
Every graph $G$ satisfies $\pin(\s(G)) = \pin(G)+1$.
\end{theorem}

\begin{proof}
Let $G_1$, $G_2$, $G_3$, $v_1$, $v_2$, $v_3$, and $w$ be as in the definition of the s-claw substitution of~$G$.

We first show that $\pin(\s(G)) \le \pin(G)+1$.
For each copy $G_i$ of~$G$, fix a path decomposition $\mathcal{Q}_i$ with minimum independence number.
Adding $v_i$ and $w$ to each bag of $\mathcal{Q}_i$ results in a path decomposition $(P^i,\beta_i)$ of the subgraph of $\s(G)$ induced by $V(G_i)\cup\{v_i,w\}$.
These three path decompositions can be combined into a single path decomposition $\mathcal{P}$ of the s-claw substitution of~$G$ in the natural way: take a new path $P$ with vertex set $V(P^1)\cup V(P^2)\cup V(P^3)$ (disjoint union) that contains each of the three paths $P^i$ as a subpath and assign to each node $x\in V(P)$ the bag $\beta_i(x)$ where $i\in \{1,2,3\}$ is such that $x$ belongs to the path $P^i$.
Since $G_i$ is isomorphic to $G$, the independence number of $\mathcal{Q}_i$ is $\pin(G)$.
Adding two adjacent vertices to each bag can increase the independence number of each bag by at most one.
It follows that the independence number of $\mathcal{P}$ is at most $\pin(G)+1$.
Thus, $\pin(\s(G)) \le \pin(G)+1$, as claimed.

Next, we show that $\pin(\s(G)) \ge \pin(G)+1$.
Suppose for a contradiction that $\pin(\s(G)) \le \pin(G)$.
Let us refer to $G_1$, $G_2$, and $G_3$ as the red, green, and blue copies of~$G$ and to their vertices as the red, green, and blue vertices, respectively.
Let $\mathcal{P} =(P,\beta)$ be a path decomposition of $\s(G)$ with minimum independence number.
Consider one of the copies of~$G$, say the red one.
Intersecting each bag of $\cal P$ with the set of red vertices yields a path decomposition $(P,\beta_1)$ of the red copy of~$G$.
Let us call a \emph{big red node} of $P$ an arbitrary but fixed node $x$ of $P$ that maximises the independence number of the subgraph of the red copy of~$G$ induced by the bag $\beta_1(x)$.
This independence number is referred to as the \emph{red independence number}.
Denoting it by $r$, we have $\pin(G) = \pin(G_1)\le r$ and therefore, since we assumed that $\pin(s(G))\le \pin(G)$, it holds that $\pin(\s(G)) \le r$.
We can similarly define a big green node $y$ and a big blue node $z$ of $P$, as well as the green and the blue independence numbers, denoted by $g$ and $b$, respectively.
By the same arguments as above, we have $\pin(\s(G)) \le g$ and $\pin(\s(G)) \le b$.

Assume that the big red and the big green nodes of $P$ coincide.
Then the subgraph of $\s(G)$ induced by the corresponding bag of $\mathcal{P}$ has independence number at least $r+g>r$, which is in contradiction with $\alpha(\mathcal{P}) = \pin(\s(G)) \le r$.
Hence, the three big nodes of $P$ are pairwise distinct and we may assume without loss of generality that they appear along $P$ in the order big red node, big green node, and big blue node.

Let $v$ be any vertex contained in the bag $\beta_1(x)$ labelling the big red node.
Then $v$ is a vertex of the red copy $G_1$ of~$G$.
In particular, the subpath $P_v$ of $P$ consisting of the nodes labelled by bags containing $v$ does not contain the big green node $y$, since otherwise the subgraph of $\s(G)$ induced by the bag $\beta(y)$ would have independence number at least $g+1$, contradicting
$\alpha(\mathcal{P}) = \pin(\s(G)) \le g$.
Since $v$ belongs to the bag labelling the big red node, which by assumption appears before the big green node, we infer that the entire path $P_v$ appears strictly before the big green node.
Furthermore, since the vertices $v$ and $v_1$ are adjacent in $G$, there exists a node $x'$ of $P$ such that $v$ and $v_1$ both belong to the bag $\beta(x')$.
Since the node $x'$ belongs to $P_v$, it appears strictly before the big green node $y$.
A similar argument as above shows that the entire subpath $P_{v_1}$ of $P$ consisting of the nodes labelled by bags of $(P,\beta)$ containing $v_1$ appears strictly before the big green node.
Since the vertices $v_1$ and $w$ are adjacent in $G$, there exists a node $x_1$ of $P$ such that $v_1$ and $w$ both belong to the bag $\beta(x_1)$.
The node $x_1$ belongs to the path $P_{v_1}$ and therefore appears strictly before the big green node $y$.

By symmetry, we can establish the existence of a node $x_3$ of $P$ that appears strictly after the big green node $y$ such that the vertices $v_3$ and $w$ both belong to the bag $\beta(x_3)$.
Since the vertex $w$ appears in a bag labelling a node before the big green node (namely $x_1$) as well as in a bag labelling a node after the big green node  (namely $x_3$), it must also appear in the bag $\beta(y)$.
Therefore, the subgraph of $\s(G)$ induced by the bag  $\beta(y)$ has independence number at least $g+1$, a contradiction with $\alpha(\mathcal{P}) = \pin(\s(G)) \le g$.
\end{proof}

Let us note that the s-claw substitution also increases $\alpha$-treedepth by~$1$.
To see that $\alpha$-treedepth indeed increases, the reader can either adapt the proof of \cref{increase-path-alpha-by-one} or, more simply, analyse which vertex of the obtained graph acts as the root of a treedepth decomposition with minimum independence number.
In fact, this latter proof approach shows that $\alpha$-treedepth is increased exactly by~$1$ also by the simpler operation of \emph{$P_5$ substitution}, which is defined analogously to the s-claw substitution, but requiring only two copies of~$G$ instead of three.
However, in contrast with $\alpha$-treedepth and $\alpha$-pathwidth, we are not aware of any transformation with the same property for $\alpha$-treewidth (that does not rely on computing the $\alpha$-treewidth of~$G$).

\begin{remark}\label{rem:suchy}
There are other graph transformations increasing $\alpha$-pathwidth by exactly~$1$.
For example, as observed by Ond\v{r}ej Such\'y (personal communication, 2024), a result analogous to \Cref{increase-path-alpha-by-one} also holds for the operation of \emph{net substitution}, which is defined similarly as the s-claw substitution but with respect to the line graph of the subdivided claw.
\end{remark}

\subsection{Clique-bounded treedepth does not imply bounded \texorpdfstring{$\alpha$}{α}-pathwidth}\label{sec:bddtd-not-alpha-bdd-pw}   
Consider the sequence $\{S_n\}_{n\ge 1}$ of graphs defined as follows:
$S_1 = K_1$ and for all $n\ge 2$, the graph $S_n$ is the s-claw substitution of $S_{n-1}$.
Let $\Gamma$ denote the hereditary closure of the sequence $\{S_n\}_{n\ge 1}$, that is, $\Gamma$ consists of all graphs in the sequence and all their induced subgraphs.

\Cref{increase-path-alpha-by-one} readily implies the following.

\begin{corollary}\label{cor:gamma-unbounded-alpha-pathwidth}
The class $\Gamma$ has unbounded $\alpha$-pathwidth.
\end{corollary}

\begin{proof}
Since the $\alpha$-pathwidth of $K_1$ is $1$, \cref{increase-path-alpha-by-one} implies that $\pin(S_n) = n$, for every positive integer~$n$.
\end{proof}

On the other hand, we show next that the class $\Gamma$ has treedepth bounded by a linear function of the clique number.

\begin{lemma}\label{gamma-td-2omega}
For each graph $G \in \Gamma$, it holds that $\td(G) \leq 2\omega(G)$.
\end{lemma}

\begin{proof}
Suppose for a contradiction that there exists a graph $G\in \Gamma$ such that $\td(G) > 2\omega(G)$.
We may assume that $G$ is a minimal counterexample, that is, that
$\td(G')\le 2\omega(G')$ holds for all proper induced subgraphs~$G'$ of~$G$.
Since the treedepth of a disconnected graph is the maximum of the treedepths of its connected components, and the same holds for the clique number, the minimality of~$G$ implies that $G$ is connected.
Furthermore, since the desired inequality holds for graphs with at most one vertex, we infer that $G$ has at least two vertices.

Next, we show that $G$ does not contain any universal vertices.
Suppose for a contradiction that $u$ is a universal vertex in $G$ and let $G' = G-u$.
Then $\omega(G') = \omega(G)-1$.
By the minimality of~$G$, we obtain that $\td(G')\le 2\omega(G') = 2 \omega(G)-2$.
Note that~$G'$ may be disconnected.
Fix a rooted forest $F$ such that~$G'$ is a subgraph of the transitive closure of $F$ and the depth of $F$ is at most $2\omega(G)-2$.
Let $T$ be the rooted tree obtained from $F$ by adding a new vertex $r$ connected by an edge to every root of $F$, and making $r$ the root of $T$.
Then $T$ is a rooted tree such that $G$ is a subgraph of the transitive closure of $T$.
The depth of $T$ is the depth of $F$ plus $1$, and therefore bounded by $2\omega(G)-1$.
This shows that $G$ has treedepth at most $2\omega(G)-1$, a contradiction.

Let $n$ be the smallest positive integer such that $G$ is an induced subgraph of $S_n$.
Since $G$ has at least two vertices, $n\ge 2$.
As the graph $S_{n}$ is the s-claw substitution of the graph $S_{n-1}$, it is obtained from the disjoint union of three copies $G_1$, $G_2$, and $G_3$ of $S_{n-1}$ by adding three new vertices $v_1$, $v_2$, $v_3$, making each vertex $v_i$ adjacent to all vertices in $G_i$, and adding a fourth new vertex $w$ adjacent only to $v_1$, $v_2$ and~$v_3$.
By the minimality of $n$, we have $V(G)\not\subseteq V(G_i)$ for all $i\in \{1,2,3\}$.
If $w$ is not a vertex of~$G$, we may assume by the connectedness of~$G$ that every vertex of~$G$ belongs to $V(G_1)\cup \{v_1\}$, but then the minimality of $n$ implies that $v_1\in V(G)$ is a universal vertex in $G$, a contradiction.
This shows that $w\in V(G)$.

For each $i\in \{1,2,3\}$, let us denote by $H_i$ the subgraph of~$G$ induced by $V(G)\cap V(G_i)$.
Let $I$ be the set of all indices $i\in \{1,2,3\}$ such that $V(H_i)\neq \emptyset$.
For each $i\in I$, the graph $H_i$ is a nonempty induced subgraph of~$G$ with strictly fewer vertices than $G$, since $w\in V(G)\setminus V(H_i)$.
Thus, the minimality of~$G$ implies that $\td(H_i)\le 2\omega(H_i)$.
By the connectedness of~$G$, for each $i\in I$ the vertex $v_i$ belongs to $G$, hence, adding $v_i$ to a maximum clique in $H_i$ yields a clique in $G$, implying that $\omega(G)\ge \omega(H_i)+1$.
For each $i\in I$, let $T_i$ be a rooted tree with root $r_i$ such that $H_i$ is a subgraph of the transitive closure of $T_i$ and the depth of $T_i$ is at most $2\omega(H_i)$.
Let $T_i'$ be the rooted tree obtained from $T_i$ by adding to it two vertices $v_i$ and $w$, edges $wv_i$ and $v_ir_i$, and making $w$ a root.
Then $T_i'$ is a rooted tree such that the subgraph of~$G$ induced by $V(H_i)\cup\{v_i,w\}$ is a subgraph of the transitive closure of $T_i'$.
The depth of $T_i'$ is at most $2\omega(H_i)+2\le 2\omega(G)$.
Also, for each $i\in \overline{I} = \{1,2,3\}\setminus I$, let us denote by 
$T_i'$ the rooted tree with vertex set $V(G)\cap \{v_i,w\}$ and $w$ as a root.
Note that the subgraph of~$G$ induced by $V(G)\cap \{v_i,w\}$ is a subgraph of the transitive closure of $T_i'$, and the depth of $T_i'$ is at most $2\le 2\omega(G)$.
Thus, for each $i \in \{1,2,3\}$, we have obtained a rooted tree $T_i'$ such that the subgraph of~$G$ induced by the vertices of~$G$ that belong to $V(G_i)\cup\{v_i,w\}$ is a subgraph of the transitive closure of $T_i'$, and the depth of $T_i'$ is at most $2\omega(G)$.
These three trees have a common root, namely the vertex $w$, and their union is a rooted tree $T$ of depth at most $2\omega(G)$ such that $G$ is a subgraph of the transitive closure of $T$.
Therefore, $\td(G)\le 2\omega(G)$, a contradiction.
\end{proof}

\medskip
We now have everything ready to show the following.

\begin{theorem}\label{tw-pw-awful}
Treedepth and pathwidth are awful parameters.
\end{theorem}

\begin{proof}
For both parameters, their awfulness is witnessed by the class $\Gamma$.
Recall that, by \Cref{prop:eta}, $\pw(G)\le \td(G)$ and $\alpha\text{-}\pw(G)\le \alpha\text{-}\td(G)$, for every graph $G$.
Hence, \Cref{cor:gamma-unbounded-alpha-pathwidth} implies that the class $\Gamma$ has unbounded $\alpha$-treedepth.
Since \Cref{gamma-td-2omega} states that $\Gamma$ is $(\td,\omega)$-bounded, we conclude that treedepth is not awesome.
Similarly, since every graph class that is $(\td,\omega)$-bounded is also $(\pw,\omega)$-bounded, the class $\Gamma$ is $(\pw,\omega)$-bounded, which, together with \Cref{cor:gamma-unbounded-alpha-pathwidth}, implies that pathwidth is not awesome.
\end{proof}

The fact that pathwidth is awful is particularly interesting in view of the aforementioned result of Hajebi~\cite{hajebi2025polynomialboundspathwidth}, showing that if a graph class has clique-bounded pathwidth, then this can always be certified with a \textsl{polynomial} binding function; recall that the analogous implication fails for treewidth~\cite{CT24}.

Let us also remark that the class $\Gamma$ cannot be used to obtain a simpler proof of the fact that treewidth is awful.
Indeed, a simple inductive argument shows that every graph $S_n$ is chordal (that is, it does not contain induced cycles of length more than $3$) and $P_6$-free.
Hence, $\Gamma$ is a subclass of the class of chordal graphs, which have $\alpha$-treewidth at most~$1$ (see~\cite{dallard2022firstpaper}).
Since $\Gamma$ is in fact a subclass of the class of $\{P_6,C_4,C_5,C_6\}$-free graphs, this shows that for graph classes defined by finitely many forbidden induced subgraphs, bounded $\alpha$-treewidth is not equivalent to bounded $\alpha$-pathwidth.\footnote{One could obtain a simpler separating example, namely the class of $\{P_5,C_4,C_5\}$-free graphs, by applying \Cref{rem:suchy}.}
This is in contrast with the fact that for classes defined by finitely many forbidden induced subgraphs, bounded treewidth is equivalent to bounded pathwidth (see~\cite{MR4385180,MR4596345}).

\section{Awesome graph parameters}\label{sec:awesome}  

In the previous section we identified a number of awful parameters. 
Many other parameters of this type are waiting to be explored. 
In order to show that the world of graph parameters is not so awful, in this section we turn to awesome parameters.

\subsection{Order, independence number, maximum degree, and clique number}    

We show here, for the sake of completeness, that certain graph parameters are awesome, as a simple consequence of their definitions.
Recall that we identify the order, as a graph parameter, with any of its trivial basic hyperparameterisations, $\opt\in \{\minmax,\maxmin\}$ and $\mathcal{F}_G = \big\{\{V(G)\}\big\}$, for all graphs $G$.
Furthermore, we identify the independence number with its canonical hyperparameterisation given by $\opt = \minmax$ and $\mathcal{F}_G = \big\{\{X \colon X$ is an independent set in $G\}\big\}$, for all graphs $G$.

\begin{proposition}\label{prop:alpha-and-n-are-awesome}
The order and the independence number are awesome graph parameters.
\end{proposition}

\begin{proof}
Note that the independence variant of either the order or the independence number is the independence number.
Hence, it suffices to show that for every hereditary graph class that has clique-bounded order (resp., independence number), the independence number is bounded.
In fact, since  clique-bounded order implies  clique-bounded independence number, it suffices to show that for every hereditary graph class that has clique-bounded independence number, the independence number is bounded.
But this follows from the fact that if a hereditary graph class~$\mathcal{G}$ has unbounded independence number, then~$\mathcal{G}$ contains arbitrarily large edgeless graphs, and, hence, cannot have  clique-bounded independence number.
\end{proof}

Recall that we identify the maximum degree graph parameter with its canonical hyperparameterisation given by $\opt = \minmax$ and $\mathcal{F}_G = \big\{\{N_G(v) \colon v\in V(G)\}\big\}$.

\begin{proposition}\label{prop:Delta-is-awesome}
The maximum degree is an awesome graph parameter.
\end{proposition}

\begin{proof}
It suffices to show that for every hereditary graph class that has clique-bounded maximum degree, the independence variant of the maximum degree is also bounded.
Let~$\mathcal{G}$ be a hereditary graph class that has $\omega$-bounded $\Delta$, and let $f$ be a nondecreasing function such that $\Delta(G)\le f(\omega(G))$ for all $G\in \mathcal{G}$.
We claim that~$\mathcal{G}$ has bounded $\alpha\text{-}\Delta$, with $\alpha\text{-}\Delta(G)\le f(2)$ for all $G\in \mathcal{G}$.
Let $G$ be an arbitrary graph in~$\mathcal{G}$.
The inequality $\alpha\text{-}\Delta(G)\le f(2)$ is equivalent to the statement that for all vertices $v\in V(G)$, we have $\alpha(G[N(v)])\le f(2)$, or, equivalently, for all vertices $v\in V(G)$ and all independent sets $I$ in the subgraph of~$G$ induced by $N(v)$, we have $|I|\le f(2)$.
Fix a vertex $v\in V(G)$ and an independent set $I$ in the subgraph of~$G$ induced by $N(v)$.
Note that the subgraph~$G'$ of~$G$ induced by $\{v\}\cup I$ is isomorphic to the star graph $K_{1,p}$, where $p = |I|$.
Since~$\mathcal{G}$ is hereditary, the graph~$G'$ belongs to~$\mathcal{G}$ and consequently $\Delta(G')\le f(\omega(G'))\le f(2)$, where the right inequality follows from the facts that $\omega(G')\le 2$ and $f$ is nondecreasing.
Since $\Delta(G') = p$, this shows that $p\le f(2)$ and completes the proof.
\end{proof}

Another simple reason for a graph parameter to be awesome is that the independence variant of the parameter is bounded for all graphs.

\begin{lemma}\label{lem:bounded-alpha-properties}
Let $\rho = (\mathtt{card}, \opt, \mathcal{F})$ be a basic hyperparameter such that there exists an integer $k$ such that one of the following holds:
\begin{itemize}
    \item $\mathrm{opt}= \mathrm{minmax}$ and for every graph $G$, there exists a hypergraph $H\in \mathcal{F}_G$ such that for every hyperedge $X\in H$, we have $\alpha(G[X])\le k$, or
    \item $\mathrm{opt}= \mathrm{maxmin}$ and for every graph $G$ and every hypergraph $H\in \mathcal{F}_G$, there exists a hyperedge $X\in H$ such that $\alpha(G[X])\le k$.    
\end{itemize}
Then, $\rho$ is awesome.
\end{lemma}

\begin{proof}
The assumptions on $\rho$ imply that $\alpha\text{-}\rho(G)\le k$ for every graph $G$, implying, in particular, that the class of all graphs has bounded $\alpha\text{-}\rho$.
Hence, every graph class has bounded $\alpha\text{-}\rho$ and, consequently, by \cref{prop:main}, every graph class has clique-bounded $\rho$.
Thus, in this case, clique-boundedness of $\rho$ is equivalent to boundedness of $\alpha\text{-}\rho$.
\end{proof}

\begin{corollary}\label{cor:bounded-alpha-properties}
Let $\rho = (\mathtt{card}, \opt, \mathcal{F})$ be a basic hyperparameter such that, for some integer $k$, every graph $G$, every hypergraph $H\in \mathcal{F}_G$, and every hyperedge $X\in H$ satisfy $\alpha(G[X])\le k$.
Then, $\rho$ is awesome.
\end{corollary}

\Cref{cor:bounded-alpha-properties} implies that the clique number is awesome, and, more generally, so is the parameter defined as the maximum number of vertices in an induced subgraph with independence number at most~$k$, for any fixed $k\in \mathbb{N}$.

\begin{sloppypar}
\subsection{Modulators}
\end{sloppypar}

\begin{sloppypar}
For a graph parameter $\rho$, an integer $c$, and a graph $G$, a set $S\subseteq V(G)$ is a \emph{$(\rho,c)$-modulator} if $\rho(G-S)\le c$.
We denote by $\rhocm(G)$ the \emph{$(\rho,c)$-modulator number} of~$G$, defined as the minimum cardinality of a $(\rho,c)$-modulator in $G$.
For example, if $\rho$ is the treewidth, then we obtain the notion of $(\tw,c)$-modulators (see, e.g.,~\cite{DBLP:conf/focs/FominLMS12,DBLP:conf/icalp/LokshtanovR0SZ18,MR3149081,MR3333298}).
In particular, $(\tw,1)$-modulators are precisely the vertex covers of~$G$, while $(\tw,2)$-modulators are precisely the feedback vertex sets.\footnote{Recall that our definition of treewidth disregards the usual ``$-1$''.}
As another example, if $\rho$ is the chromatic number $\chi$, then $(\chi,2)$-modulators are precisely the odd cycle transversals.
\end{sloppypar}

Note that for every graph parameter $\rho$ and an integer $c$, the $(\rho,c)$-modulator number satisfies the conditions of \cref{hyperpar-optimisation}, hence, it admits a canonical hyperparameterisation and corresponds to a basic hyperparameter.
This means for every annotated graph parameter $\lambda$, the $\lambda$-variant of the $(\rho,c)$-modulator number, denoted $\lambda\text{-}\rhocm$, is well-defined: Given a graph $G$, the value of $\lambda\text{-}\rhocm(G)$ equals the minimum value of $\lambda(G,S)$, over all $(\rho,c)$-modulators $S$ of~$G$.
In particular, the independence variant of $\rhocm$ is well-defined and one can ask if $\rhocm$ is awesome.\\

\subsubsection{Heaviness} 

In this section, we identify a sufficient condition for a $(\rho,c)$-modulator number to be awesome (\Cref{modulator is awesome}).
The main ingredient for the proof is \cref{lem:tau-k-omega-alpha}, which is a consequence of the following more general lemma that we apply also to answer a question from \cite{HMV25} (see \Cref{thm:OK-free}). 
   
\begin{lemma}\label{lem:tau-k-omega-alpha-general}
    Let $\rho$ be a graph parameter such that there exists a nondecreasing function $h$ such that $\omega(G) \le h(\rho(G))$ holds for all graphs $G$.
    Let $g : \mathbb{N} \rightarrow \mathbb{N}$ be a nondecreasing function, $c$ be an integer, and let~$\mathcal{G}$ be a hereditary graph class such that there exists a nondecreasing function $f \colon \mathbb{N} \rightarrow \mathbb{N}$ such that $\rhocm(G)\le f(\omega(G))\cdot g(|V(G)|)$ holds for all graphs $G\in \mathcal{G}$.
    Then, for every graph $G\in \mathcal{G}$, every minimum $(\rho, c)$-modulator $S$ satisfies $\alpha(G[S])\le f(h(c)+1)\cdot g(|V(G)|)$.
\end{lemma}
\begin{proof}
    Let $G\in \mathcal{G}$ and let $S$ be a minimum $(\rho,c)$-modulator in $G$.
    Suppose for a contradiction that $\alpha(G[S])> f(h(c)+1) \cdot g(|V(G)|)$.
    Then, $G$ has an independent set $I$ such that $I\subseteq S$ and $|I|>f(h(c)+1) \cdot g(|V(G)|)$.
    Let~$G'$ be the subgraph of~$G$ induced by $(V(G)\setminus S)\cup I$.
    Since $\rho(G-S)\le c$ and $h$ is nondecreasing, the clique number of the graph $G-S$ is at most $h(c)$.
    Consequently, $\omega(G')\le \omega(G-S) + \omega(G[I])\le h(c)+1$.
    Since the class~$\mathcal{G}$ is hereditary, the graph~$G'$ belongs to~$\mathcal{G}$. 
    It follows that $\rhocm(G')\le f(\omega(G')) \cdot g(|V(G')|) \le f(h(c)+1) \cdot g(|V(G)|)$, where the last inequality follows from the inequality $\omega(G')\le h(c)+1$ and the fact that both $f$ and $g$ are nondecreasing.
    Let $S'$ be a $(\rho, c)$-modulator in~$G'$ such that $|S'|\le f(h(c)+1) \cdot g(|V(G)|)$.
    Note that 
    \[
        |S'|\le f(h(c)+1) \cdot g(|V(G)|)<|I|\,,
    \] 
    which implies 
    \[
        |V(G'-S')| = |V(G-S)|+|I|-|S'| > |V(G-S)|\,.
    \]
    Since $S'$ is a $(\rho, c)$-modulator in~$G'$, the graph $G'-S'$ satisfies $\rho(G'-S')\le c$.
    It follows that the set $\widehat{S}:= V(G)\setminus V(G'-S')$ is a $(\rho, c)$-modulator in $G$.
    Its cardinality satisfies
    \[
        |\widehat{S}| = |V(G)|-|V(G'-S')| < |V(G)|-|V(G-S)| = |S|\,,
    \]
    contradicting the fact that $S$ is a minimum $(\rho,c)$-modulator in $G$.
\end{proof}

Setting $g$ to be the constant function 1 in \cref{lem:tau-k-omega-alpha-general} yields the following corollary.

\begin{corollary}\label{lem:tau-k-omega-alpha}
Let $\rho$ be a graph parameter such that there exists a nondecreasing function $h$ such that $\omega(G) \le h(\rho(G))$ holds for all graphs $G$.
Let $c$ be an integer and~$\mathcal{G}$ be a hereditary graph class that has clique-bounded $\rhocm$ and let $f : \mathbb{N} \rightarrow \mathbb{N}$ be a nondecreasing function such that $\rhocm(G)\le f(\omega(G))$ holds for all $G\in \mathcal{G}$.
Then, for every graph $G\in \mathcal{G}$, every minimum $(\rho, c)$-modulator $S$ satisfies $\alpha(G[S])\le f(h(c)+1)$.
\end{corollary}

A graph parameter $\rho$ is said to be \emph{heavy} if it is monotone under induced subgraphs and unbounded on the class of complete graphs.

\begin{theorem}\label{modulator is awesome}
For every heavy graph parameter $\rho$ and integer $c$, the $(\rho,c)$-modulator number $\rhocm$ is awesome.
\end{theorem}

\begin{proof}
Let $c$ and $\rho$ be as in the statement.
For an integer $x$, let $\mathcal{C}_x$ denote the class of all graphs $G$ such that $\rho(G)\le x$.
We first show that the clique number of graphs in $\mathcal{C}_x$ is bounded by some function of $x$.
Suppose this is not the case. 
Then there exists some integer $x$ such that graphs in $\mathcal{C}_x$ have arbitrarily large clique numbers.
Note that the class $\mathcal{C}_x$ is hereditary since $\rho$ is monotone under induced subgraphs. 
Hence, the class $\mathcal{C}_x$ contains arbitrarily large complete graphs, implying that $\rho$ is unbounded on $\mathcal{C}_x$, a contradiction.

Since the clique number of graphs in $\mathcal{C}_x$ is bounded by some function of $x$, we can now define a function $h$ that allows us to apply \cref{lem:tau-k-omega-alpha}: for every integer $x$, we define the value of $h(x)$ to be the maximum clique number of all graphs in $\mathcal{C}_x$ (and $0$ if $\mathcal{C}_x$ is empty).
Since $x\le y$ implies $\mathcal{C}_x\subseteq\mathcal{C}_y$, the function $h$ is nondecreasing.
Furthermore, for every graph $G$, we have $\omega(G) \le h(\rho(G))$.

To prove the statement, it suffices to show that every hereditary graph class that has clique-bounded $\rhocm$ also has bounded $\alpha\text{-}\rhocm$.
Let~$\mathcal{G}$ be a hereditary graph class that has clique-bounded $\rhocm$, and let $f$ be a nondecreasing function such that $\rhocm(G)\le f(\omega(G))$ for all $G\in \mathcal{G}$.
By \cref{lem:tau-k-omega-alpha}, for every graph $G\in \mathcal{G}$, every minimum $(\rho, c)$-modulator $S$ satisfies $\alpha(G[S])\le f(h(c)+1)$.
Hence, $\alpha\text{-}\rhocm(G)\le f(h(c)+1)$, for every graph $G\in \mathcal{G}$, in other words,~$\mathcal{G}$ has bounded $\alpha\text{-}\rhocm$.
\end{proof}

As a consequence of \cref{modulator is awesome}, we get the following result, which highlights the fact that the $(\rho,c)$-modulator numbers can be awesome, although $\rho$ may not be (for instance when $\rho$ is treewidth, pathwidth, or treedepth).

\begin{corollary}\label{cor:modulator-examples}
For every $\rho\in\{\omega,\chi,\tw,\pw,\td\}$ and integer $c$, the $(\rho,c)$-modulator number $\rhocm(G)$ is an awesome graph parameter.
\end{corollary}

\begin{sloppypar}
Since for any graph $G$, the $(\tw,1)$-modulators, $(\tw,2)$-modulators, and the \hbox{$(\chi,2)$-modulators} of~$G$ are precisely its vertex covers,\footnote{It can also be observed that the vertex covers of~$G$ are precisely its $(\td,1)$-modulators.} feedback vertex sets, and odd cycle transversals, respectively, we obtain the following consequence of \Cref{cor:modulator-examples}.
\end{sloppypar}

\begin{corollary}\label{cor:awesome-examples}
The vertex cover number, the feedback vertex set number, and the odd cycle transversal number are awesome graph parameters.
\end{corollary}

Other graph parameters captured by \cref{cor:modulator-examples} include various  generalisations of the vertex cover number, for example the \textsl{vertex $3$-path cover number} (see, e.g.,~\cite{MR3056999,MR2806015}), which is equal to the minimum size of a $(\Delta,1)$-modulator, or, for any $k>1$, the minimum size of a \textsl{vertex $k$-separator} (see, e.g.,~\cite{MR3987192}, as well as~\cite{MR4927577,MR3567632}), which is defined as the minimum cardinality of a $(\rho,k-1)$-modulator in $G$, where $\rho(G)$ is the maximum order of a component of~$G$.

Further applications of \cref{cor:modulator-examples} are given in \cref{sec:applications}.

\subsubsection{Inheritability and strong inheritability}

We now introduce a condition on basic hyperparameters, which we call inheritability, that allows for a transfer of the class properties from the modulators of a graph parameter to the parameter itself, as we will show in the next subsection.
A~basic hyperparameter $\rho$ is said to be \emph{inheritable} if for each $\lambda\in \{\mathtt{card},\alpha\}$, every graph $G$, and every set $S\subseteq V(G)$, it holds that $\lambda\text{-}\rho(G)\le \rho(G-S) + \lambda(G,S)$.

The following structural condition is sufficient for inheritability.
A basic hyperparameter $\rho=(\mathtt{card},\opt,\mathcal{F})$ is said to be \emph{strongly inheritable} if one of the following conditions holds:
\begin{itemize}
    \item $\mathrm{opt}= \mathrm{minmax}$ and for every graph $G$, every set $S\subseteq V(G)$, and every hypergraph $H'\in \mathcal{F}_{G-S}$, there exists a hypergraph $H\in \mathcal{F}_{G}$ such that $\{X\setminus S\colon X\in H\} \subseteq H'$, that is, $X\setminus S\in H'$ for all $X\in H$;
    \item $\mathrm{opt}= \mathrm{maxmin}$ and for every graph $G$, every set $S\subseteq V(G)$, and every hypergraph $H\in \mathcal{F}_{G}$, there exists a hypergraph $H'\in \mathcal{F}_{G-S}$ such that $H'\subseteq \{Y\setminus S\colon Y\in H\}$, that is, for all $X\in H'$ there exists some $Y\in H$ such that $X = Y\setminus S$.
\end{itemize}

Modulators, for instance, give rise to strongly inheritable hyperparameters.

\begin{proposition}\label{obs:inheritable-2}
For every graph parameter $\rho$ and integer $c$, the $(\rho,c)$-modulator number $\rhocm$ is strongly inheritable.
\end{proposition}

\begin{proof}
Let $(\lambda, \mathrm{opt}, \mathcal{F})$ be the canonical hyperparameterisation of $\rhocm$ given by \cref{hyperpar-optimisation}.
Then $\opt = \maxmin$ and $\mathcal{F}_G = \big\{\{X\subseteq V(G) \colon \rho(G-X)\le c\}\big\}$ for all graphs $G$.
Hence, $\rhocm$ is strongly inheritable if and only if for every graph $G$, every set $S\subseteq V(G)$, and every set $X'\subseteq V(G-S)$ such that $\rho(G-S-X')\le c$, there exists a set $X\subseteq V(G)$ such that $\rho(G-X)\le c$ and $X' = X\setminus S$.
But this is clear, as we can take $X = X'\cup S$.
\end{proof}

Several other graph parameters of interest are also strongly inheritable. 

\begin{proposition}\label{obs:inheritable}
Let $\rho\in\{\omega,\chi,\tw,\pw,\td\}$.
Then, $\rho$ is strongly inheritable.
\end{proposition}

\begin{proof}
Note that all these parameters are $\minmax$ parameters (for the clique number, this follows from \Cref{hyperpar-optimisation}).
Clique number is strongly inheritable because for every graph $G$ and a set $S\subseteq V(G)$, any clique in the graph $G-S$ is a clique in $G$.
Chromatic number is strongly inheritable because any proper colouring of the graph $G-S$ can be extended to a proper colouring of the graph $G$.
Treewidth is strongly inheritable as a consequence of the fact that for every graph $G$ and $S\subseteq V(G)$, if $\mathcal{T} =(T,\beta)$ is a tree decomposition of $G-S$, then $(T,\beta_G)$, where $\beta_G(x) = \beta(x)\cup S$ for all nodes $x$ of $T$, is a tree decomposition of~$G$.
The argument for pathwidth is similar.
Treedepth is strongly inheritable since for every graph $G$ and $S\subseteq V(G)$, if $F$ is a treedepth decomposition of $G-S$, then the rooted forest obtained from a rooted path $P$ on $S$ by adding an edge from the sink of $P$ to each of the roots of $F$ is a treedepth decomposition of~$G$.
\end{proof}

Following Adler~\cite{adler2006width}, we say that an annotated graph parameter $\lambda$ is:\label{def:monotone-wekly-submod-tame}
\begin{itemize}
    \item \emph{monotone} if for all graphs $G$ and all $X\subseteq Y\subseteq V(G)$, we have $\lambda(G,X)\le \lambda(G,Y)$;
    \item \emph{weakly submodular} if for all graphs $G$ and finite collections $(X_i)_{i\in I}$ of subsets of $V(G)$, we have $\lambda\left(G,\bigcup_{i\in I}X_i\right)\le \sum_{i\in I}\lambda(G,X_i)$;
    \item \emph{tame} if for all graphs $G$ and $v\in V(G)$, we have $\lambda(G,\{v\})\le 1$.
\end{itemize}
Note, for instance, that each $\lambda\in \{\texttt{card},\alpha\}$ is monotone, weakly submodular, and tame.

\begin{observation}\label{at-most-card}
Let $\lambda$ be a weakly submodular and tame annotated graph parameter.
Then \hbox{$\lambda(G,X)\le |X|$} for every non-null graph $G$ and every set $X\subseteq V(G)$.
\end{observation}

As a side remark, let us observe that for weakly submodular and tame annotated graph parameters $\lambda$, the chain of inequalities ${\lambda\text{-}\tw(G)\le \lambda\text{-}\pw(G)\le \lambda\text{-}\td(G)}$ from \Cref{prop:eta} can be further extended to include the $\lambda$-variant of vertex cover number.

\begin{proposition}\label{prop:lambda-td-vc}
    For every graph $G$ and every weakly submodular and tame annotated graph parameter $\lambda$, it holds that
    \[
    \lambda\text{-}\td(G) \leq  \lambda\text{-}\vc(G)+1.
    \]
\end{proposition}

\begin{proof}
Let $G$ be a graph.
By the definition of $\lambda\text{-}\vc$, we have \[\lambda\text{-}\vc(G) = \min_{X\in H_{\vc}(G)}\lambda(G,X)\,,\] where $H_{\vc}(G)$ is the hypergraph over $G$ whose hyperedges are vertex covers of~$G$.
Similarly, 
\[\lambda\text{-}\td(G) = \lambda\text{-}\minmax(G,\mathcal{F}_G^{\td})= \min_{H\in \mathcal{F}_G^{\td}}\max_{X\in H}\lambda(G,X)\,,\] where $\mathcal{F}_G^{\td}$ is the family of all hypergraphs over $G$ whose hyperedges are vertex sets of root-to-leaf paths of some treedepth decomposition of~$G$.
Let $C = \{v_1,\ldots, v_k\}$ be a vertex cover of~$G$ such that $\lambda(G,C) = \lambda\text{-}\vc(G)$.
Let $F$ be the rooted forest with vertex set $V(G)$ with directed edges $\{(v_i,v_{i+1})\colon 1\le i\le k-1\}\cup \{(v_k,v)\colon v\in V(G)\setminus C\}$. 
Since $V(G)\setminus C$ is an independent set in $G$, the graph $G$ is a subgraph of the transitive closure of $F$.
Therefore, the hypergraph $H$ over $G$ whose hyperedges are precisely the sets $C\cup \{v\}$ for all $v\in V(G)\setminus C$, belongs to $\mathcal{F}_G^{\td}$.
Furthermore, 
\begin{align*}
\max_{X\in H}\lambda(G,X) &= \max_{v\in V(G)\setminus C} \lambda(G,C\cup \{v\}) \\
&\le 
 \max_{v\in V(G)\setminus C} \left(\lambda(G,C)+\lambda(G,\{v\})\right)\\ 
 &\le \lambda(G,C)+1\\
 &= \lambda\text{-}\vc(G)+1\,,
\end{align*}
where the two inequalities follow from the facts that $\lambda$ is weakly submodular and tame, respectively.
Consequently, $\lambda\text{-}\td(G)\le \max_{X\in H}\lambda(G,X)\le \lambda\text{-}\vc(G)+1$, as claimed.
\end{proof}

Next, we show that strong inheritability implies inheritability. 

\begin{lemma}\label{inheritable-lambda-rho at most}
Let $\rho$ be a strongly inheritable hyperparameter.
Then, for every monotone, weakly submodular, and tame annotated graph parameter $\lambda$, every graph $G$, and every set $S\subseteq V(G)$, it holds that $\lambda\text{-}\rho(G)\le \rho(G-S) + \lambda(G,S)$.
In particular, $\rho$ is inheritable.
\end{lemma}

\begin{proof}
Let $\rho=(\mathtt{card},\opt,\mathcal{F})$ be a strongly inheritable  hyperparameter and let $\lambda$, $G$, and $S$ be as in the lemma.
Let $c = \rho(G-S)$.

Consider first the case when $\mathrm{opt}= \mathrm{minmax}$.
Then
\[c = \min_{H \in \mathcal{F}_{G-S}} \max_{X \in H} |X|\,.\]
In particular, there exists a hypergraph $H^* \in \mathcal{F}_{G-S}$ such that $\max_{X \in H^*} |X| = c$.
Since $\rho$ is strongly inheritable, there exists a hypergraph $\widehat H\in \mathcal{F}_{G}$ such that $\{X\setminus S\colon X\in \widehat H\} \subseteq H^*$, that is, $X\setminus S\in H^*$ for all $X\in \widehat H$.
This implies that 
\[\min_{H \in \mathcal{F}_{G}} \max_{X \in H} |X\setminus S| \le \max_{X \in \widehat{H}} |X\setminus S|\le \max_{X \in H^*} |X| = c\,.\]
Consequently, 
\begin{alignat*}{2}
\lambda\text{-}\rho(G) & = \min_{H \in \mathcal{F}_{G}} \max_{X \in H} \lambda(G,X) \\
&\le \min_{H \in \mathcal{F}_{G}} \max_{X \in H} \Big(\lambda(G,X \cap S) + \lambda(G,X\setminus S)\Big)\\ \tag*{\text{(since $\lambda$ is weakly submodular)}} \\
&\le \lambda(G,S) + \min_{H \in \mathcal{F}_{G}} \max_{X \in H} |X\setminus S|\\ \tag*{\text{(by monotonicity and \Cref{at-most-card})}} \\
&\le \lambda(G,S) +c\,,\\
&= \lambda(G,S) +\rho(G-S)\,,
\end{alignat*}
yielding the desired inequality.

Suppose now that $\mathrm{opt}= \mathrm{maxmin}$.
Then 
\[c = \rho(G-S) =\max_{H \in \mathcal{F}_{G-S}} \min_{X \in H} |X|\,.\]
In particular, for all hypergraphs $H \in \mathcal{F}_{G-S}$ it holds that $\min_{X \in H} |X| \le c$.
Consider an arbitrary hypergraph $H\in  \mathcal{F}_{G}$.
Since $\rho$ is strongly inheritable, there exists a hypergraph $H'\in \mathcal{F}_{G-S}$ such that $H'\subseteq \{X\setminus S\colon X\in H\}$.
Let $X_{H'}$ be a hyperedge in $H'$ such that $|X_{H'}|\le c$.
Then, there exists a hyperedge $X_H\in H$ such that $X_{H'} = X_H\setminus S$, and, hence,
\[\min_{X \in H} |X\setminus S|\le |X_H\setminus S| = |X_{H'}|\le c\,.\]
Since $H$ is arbitrary, this implies that 
\[\max_{H \in \mathcal{F}_{G}} \min_{X \in H} |X\setminus S| \le c\,.\]
Consequently, 
\begin{alignat*}{1}
\lambda\text{-}\rho(G) &= \max_{H \in \mathcal{F}_{G}} \min_{X \in H} \lambda(G,X)\\
&\le \max_{H \in \mathcal{F}_{G}} \min_{X \in H} \Big(\lambda(G,X \cap S) + \lambda(G,X\setminus S)\Big)\\ \tag*{\text{(since $\lambda$ is weakly submodular)}}\\
&\le \lambda(G,S) + \max_{H \in \mathcal{F}_{G}} \min_{X \in H} |X\setminus S|\\ \tag*{\text{(by monotonicity and \Cref{at-most-card})}}\\
&\le \lambda(G,S) +c\\
&= \lambda(G,S) + \rho(G-S)\,
\end{alignat*}
as claimed.

Hence, in both cases, $\lambda\text{-}\rho(G)\le \lambda(G,S) + \rho(G-S)$.
Since each $\lambda\in \{\mathtt{card},\alpha\}$ is  monotone, weakly submodular, and tame, it follows that $\rho$ is inheritable.
\end{proof}

As an immediate consequence of \cref{obs:inheritable,obs:inheritable-2,inheritable-lambda-rho at most}, we obtain the following.

\begin{corollary}\label{cor:inheritable}
For every graph parameter $\rho$ and integer $c$, the $(\rho,c)$-modulator number $\rhocm$ is inheritable.
Furthermore, if $\rho\in\{\omega,\chi,\tw,\pw,\td\}$, then $\rho$ is inheritable.
\end{corollary}

\subsubsection{From modulators to parameters and back}

We now compare three class properties defined in terms of a basic hyperparameter~$\rho$.
More precisely, we compare the properties bounded $\rho$, bounded $\alpha$-$\rho$, and clique-bounded $\rho$ (recall also the respective notations $\mathcal{B}_{\rho}$, $\mathcal{B}_{\alpha\text{-}\rho}$, and $\mathcal{C}_{\rho}$ from \Cref{sec:antiRamsey}) with the analogous properties for the $(\rho,c)$-modulator numbers.

Note that if a graph class~$\mathcal{G}$ belongs to $\mathcal{B}_{\rho}$, say with $\rho(G)\le c$ for all induced subgraphs $G$ of graphs in~$\mathcal{G}$, then every such graph $G$ satisfies $\rhocm(G) = 0$; in particular,~$\mathcal{G}$ belongs to the class property $\mathcal{B}_{\rhocm}$.
This implies that the class property $\mathcal{B}_{\rho}$ is contained in the union, over all integers $c$, of the class properties $\mathcal{B}_{\rhocm}$. 
For later use, we state this as an \lcnamecref{obs:class-property}.
\begin{observation}\label{obs:class-property}
Let $\rho$ be a basic hyperparameter. 
Then, $\mathcal{B}_{\rho} \subseteq \bigcup_{c}\mathcal{B}_{\rhocm}$.
\end{observation}
Note also that if $c \le c'$, then $\mathcal{B}_{\rhocm}\subseteq \mathcal{B}_{\mu_{\rho,c'}}$, so the \emph{limiting property} $\lim_{c\to \infty}\mathcal{B}_{\rhocm}$ is well-defined and equals to the union, over all integers $c$, of the class properties $\mathcal{B}_{\rhocm}$.
We show that whenever the basic hyperparameter is inheritable, the limiting property is actually equal to the class property $\mathcal{B}_{\rho}$. 
On the other hand, we show that the analogous identities fail for $\mathcal{B}_{\alpha\text{-}\rho}$ and $\mathcal{C}_{\rho}$.
More precisely, we show that for every inheritable hyperparameter $\rho$, the following holds (see \Cref{fig questions}): 
\begin{itemize}
    \item $\bigcup_{c}\mathcal{B}_{\rhocm} = \mathcal{B}_{\rho}$ (see~\Cref{obs:class-property,thm:alpha-rho at most alpha-rhocm}).
   \item $\bigcup_{c}\mathcal{B}_{\alpha\text{-}\rhocm}\subseteq \mathcal{B}_{\alpha\text{-}\rho}$ (see \cref{thm:alpha-rho at most alpha-rhocm}), but the inclusion may be strict (see~\cref{prop:alpha-mu-rho-c-bounded-strict}).
    \item $\bigcup_{c}\mathcal{C}_{\rhocm}\subseteq \mathcal{C}_{\rho}$ (see \cref{thm:alpha-rho at most alpha-rhocm}), but the inclusion may be strict (see \cref{cor:clique-bounded-strict}).
\end{itemize}

\begin{figure}[h!]
    \centering
    \begin{tikzpicture}[yscale=1.25]
    \crefname{proposition}{Prop.}{Prop.}
    \crefname{theorem}{Thm.}{Thm.}
    
    \usetikzlibrary{arrows.meta}
    \usetikzlibrary{calc}
        \tikzset{every node/.style={rectangle, draw, very thick, inner sep = 5pt, outer sep = 2pt, align = center, rounded corners=2.pt}}

        \node[text width = 3.5cm, align = center] (bounded-rhocm-some-c) at (0,0) {\strut bounded $\rhocm$, for some $c$};
        \node[text width = 3.5cm, align = center, right = 3cm of bounded-rhocm-some-c] (bounded-rho) {\strut bounded $\rho$};
        
        \node[text width = 3.5cm, align = center, below = 1.65cm of bounded-rhocm-some-c] (bounded-alpha-rhocm-some-c) {\strut bounded $\alpha\text{-}\rhocm$, for some $c$};
        \node[text width = 3.5cm, align = center, right = 3cm of bounded-alpha-rhocm-some-c] (bounded-alpha-rho) {\strut bounded $\alpha\text{-}\rho$};
        
        \node[text width = 3.5cm, align = center, below = 1.65cm of bounded-alpha-rhocm-some-c] (clique-bounded-rhocm-some-c) {\strut clique-bounded $\rhocm$, for some $c$};
        \node[text width = 3.5cm, align = center, right = 3cm of clique-bounded-rhocm-some-c] (clique-bounded-rho) {\strut clique-bounded $\rho$};

        \draw [-{Latex[round,length=2.5mm,width=2.5mm]},very thick] ($(bounded-rhocm-some-c.east)+(0.1,0.25)$) -- node[auto,fill=none,draw=none,above,inner sep=0pt]{\strut $\star$} ($(bounded-rho.west)+(-0.1,0.25)$);
        \draw [-{Latex[round,length=2.5mm,width=2.5mm]},very thick] ($(bounded-rho.west)-(0.1,0.25)$) -- ($(bounded-rhocm-some-c.east)-(-0.1,0.25)$);
        
        \draw [-{Latex[round,length=2.5mm,width=2.5mm]},very thick] ($(bounded-alpha-rhocm-some-c.east)+(0.1,0)$) -- node[auto,fill=none,draw=none,above,inner sep=0pt]{\strut $\star$} ($(bounded-alpha-rho.west)+(-0.1,0)$);
        \draw [-{Latex[round,length=2.5mm,width=2.5mm]},very thick] ($(clique-bounded-rhocm-some-c.east)+(0.1,0)$) -- node[auto,fill=none,draw=none,above,inner sep=0pt]{\strut $\star$} ($(clique-bounded-rho.west)+(-0.1,0)$);

        \draw [-{Latex[round,length=2.5mm,width=2.5mm]},very thick] ($(bounded-rhocm-some-c.south)+(0,-0.1)$) -- ($(bounded-alpha-rhocm-some-c.north)+(0,0.1)$);
        \draw [-{Latex[round,length=2.5mm,width=2.5mm]},very thick] ($(bounded-alpha-rhocm-some-c.south)+(-0.25,-0.1)$) -- 
        node[auto,fill=none,draw=none,left]{\strut \cref{prop:main}} ($(clique-bounded-rhocm-some-c.north)+(-0.25,0.1)$);
        \draw [-{Latex[round,length=2.5mm,width=2.5mm]},very thick] ($(clique-bounded-rhocm-some-c.north)+(0.25,0.1)$) -- node[auto,fill=none,draw=none,right,text width=2.25cm]{\strut \cref{modulator is awesome} (if~$\rho$ is heavy)} ($(bounded-alpha-rhocm-some-c.south)+(0.25,-0.1)$);

        \draw [-{Latex[round,length=2.5mm,width=2.5mm]},very thick] ($(bounded-rho.south)+(0,-0.1)$) -- ($(bounded-alpha-rho.north)+(0,0.1)$);
        \draw [-{Latex[round,length=2.5mm,width=2.5mm]},very thick] ($(bounded-alpha-rho.south)+(-0.25,-0.1)$) -- node[auto,fill=none,draw=none,left]{\strut \cref{prop:main}} ($(clique-bounded-rho.north)+(-0.25,0.1)$);

        \draw [-{Latex[round,length=2.5mm,width=2.5mm]},very thick] ($(clique-bounded-rho.north)+(0.25,0.1)$) --  node[auto,fill=none,draw=none,right,text width=2cm,inner sep=0pt]{\strut iff $\rho$ is awesome} ($(bounded-alpha-rho.south)+(0.25,-0.1)$);
    \end{tikzpicture}
    \caption{Inclusions of class properties described in terms of a basic hyperparameter $\rho$.
    An arc from $A$ to $B$ means that $A \subseteq B$.
    Unlabelled arcs represent inclusions that hold for all basic hyperparameters and follow directly from the definitions.  
    The left-to-right inclusions (marked by $\star$) hold for inheritable hyperparameters and follow from \cref{thm:alpha-rho at most alpha-rhocm}.}
    \label{fig questions}
\end{figure}

The proof of the left-to-right inclusions in \Cref{fig questions} for inheritable hyperparameters is based on the following. 

\begin{observation}\label{lambda-rho at most lambda-rhocm}
Let $\rho$ be a basic hyperparameter and let $\lambda$ be an annotated graph parameter such that for every graph $G$ and every set $S\subseteq V(G)$, it holds that $\lambda\text{-}\rho(G)\le \rho(G-S) + \lambda(G,S)$.
Then, for every graph $G$ and every integer $c$, we have $\lambda\text{-}\rho(G)\le \lambda\text{-}\rhocm(G) + c$.
\end{observation}

\begin{proof}
Let $S$ be a \hbox{$(\rho,c)$-modulator} of~$G$ minimising $\lambda(G,S)$, that is, $\rho(G-S)\le c$ and $\lambda(G,S) = \lambda\text{-}\rhocm(G)$.
Then, $\lambda\text{-}\rho(G)\le \rho(G-S)+\lambda(G,S) \le c + \lambda\text{-}\rhocm(G)$, as claimed.
\end{proof}

\begin{theorem}\label{thm:alpha-rho at most alpha-rhocm}
Let $\rho$ be an inheritable hyperparameter.
Then, for every integer $c$ and every graph class~$\mathcal{G}$, the following properties hold:
\begin{enumerate}[label=(\arabic*)]
    \item\label[property]{rho-prop1} If~$\mathcal{G}$ has bounded $\rhocm$, then~$\mathcal{G}$ has bounded $\rho$.
    \item\label[property]{rho-prop2} If~$\mathcal{G}$ has bounded $\alpha\text{-}\rhocm$, then~$\mathcal{G}$ has bounded $\alpha\text{-}\rho$.
    \item\label[property]{rho-prop3} If~$\mathcal{G}$ has clique-bounded $\rhocm$, then~$\mathcal{G}$ has clique-bounded $\rho$. 
\end{enumerate}
\end{theorem}

\begin{proof}
    Let $\rho$ be as in the statement, let $c$ be an integer, and let~$\mathcal{G}$ be a graph class. 
    Applying \Cref{lambda-rho at most lambda-rhocm} to $\lambda\in \{\mathtt{card},\alpha\}$, we obtain \cref{rho-prop1,rho-prop2}.
    To prove \cref{rho-prop3}, let~$\mathcal{G}$ be a graph class that has clique-bounded $\rhocm$, and let $f$ be a function such that for every graph $G\in \mathcal{G}$ and every induced subgraph~$G'$ of~$G$, $\rhocm(G') \le f(\omega(G'))$.
    Applying \Cref{lambda-rho at most lambda-rhocm} to $\lambda = \mathtt{card}$ shows that for every $G\in \mathcal{G}$ and every induced subgraph~$G'$ of~$G$, we have $\rho(G')\le \rhocm(G') + c\le f(\omega(G')) + c$.
    Therefore, since $c$ is a constant, it follows that~$\mathcal{G}$ has clique-bounded~$\rho$.
\end{proof}

\begin{remark}\label{necessary inheritability}
In \cref{thm:alpha-rho at most alpha-rhocm}, the assumption that $\rho$ is inheritable is necessary.
Indeed, the inclusion relations between class properties marked by $\star$ in \Cref{fig questions} may fail to hold for some basic hyperparameters~$\rho$.
Consider for example the maximum degree, $\rho = \Delta$.
Let~$\mathcal{G}$ be the class of all complete bipartite graphs of the form $K_{1,n}$ and their induced subgraphs.
Then each graph $G\in \mathcal{G}$ satisfies $\mu_{\Delta,0}(G) \le 1$, hence,~$\mathcal{G}$ has bounded $\mu_{\Delta,0}$.
On the other hand, each graph $G\in \mathcal{G}$ satisfies $\omega(G) \le 2$, while the maximum degree of graphs in $G$ is unbounded, showing that the class~$\mathcal{G}$ does not have clique-bounded $\Delta$.
This implies that none of the left-to-right inclusions in \Cref{fig questions} holds for $\rho = \Delta$.
Along with \Cref{thm:alpha-rho at most alpha-rhocm}, this also implies that the maximum degree is not an inheritable parameter.
\end{remark}

\Cref{thm:alpha-rho at most alpha-rhocm} implies that for any inheritable hyperparameter $\rho$ and any constant $c$, the class property $\mathcal{B}_{\rhocm}$ is a subfamily of the class property $\mathcal{B}_{\rho}$, that is, $\bigcup_{c}\mathcal{B}_{\rhocm} \subseteq \mathcal{B}_{\rho}$.
Hence, by \Cref{obs:class-property}, the class property $\mathcal{B}_{\rho}$ equals the union, over all integers $c$, of the class properties $\mathcal{B}_{\rhocm}$.

Similarly, \Cref{thm:alpha-rho at most alpha-rhocm} implies that $\bigcup_{c}\mathcal{B}_{\alpha\text{-}\rhocm}\subseteq \mathcal{B}_{\alpha\text{-}\rho}$ and $\bigcup_{c}\mathcal{C}_{\rhocm}\subseteq \mathcal{C}_{\rho}$.
We next show that equality holds whenever $\rho$ is the $(\eta,d)$-modulator number, for any inheritable and heavy hyperparameter $\eta$ and any integer $d$.

\begin{proposition}\label{prop:modulators-equality}
Let $\eta$ be an inheritable and heavy hyperparameter, $d$ an integer, and let $\rho$ be the $(\eta,d)$-modulator number.
Then, $\rho$ is inheritable and heavy, and $\mathcal{B}_{\alpha\text{-}\rho} = \bigcup_{c}\mathcal{B}_{\alpha\text{-}\rhocm}$ and $\mathcal{C}_{\rho} = \bigcup_{c}\mathcal{C}_{\rhocm}$.
\end{proposition}

\begin{proof}
First we show that $\rho$ is heavy.
Let $G$ be a graph and let~$G'$ be an induced subgraph of~$G$.
Fix a minimum-cardinality set $S\subseteq V(G)$ such that $\eta(G-S) \le d$.
Let $S' \coloneqq V(G')\cap S$.
Since $G'-S'$ is an induced subgraph of $G-S$ and $\eta$ is monotone under induced subgraphs, we have $\eta(G'-S')\le \eta(G-S) \le d$.
Hence, $\rho(G')\le |S'|\le |S| = \rho(G)$, implying that $\rho$ is monotone under induced subgraphs.
To show that $\rho$ is heavy, it remains to show that $\rho$ is unbounded on the class of complete graphs.
Suppose for a contradiction that there exists an integer $c$ such that $\rho(K)\le c$ for every complete graph $K$.
Since $\eta$ is heavy, there exists a complete graph $K$ such that $\eta(K) > c+d$.
As $\rho(K)\le c$, there exists a set $S\subseteq V(K)$ such that $|S|\le c$ and $\eta(K-S)\le d$.
Since $\eta$ is inheritable, we have  $\eta(K)\le \eta(K-S) + |S|\le d+c$, a contradiction.
Hence, $\rho$ is heavy.
By~\Cref{cor:inheritable}, $\rho$ is inheritable.

We now show the inclusion $\mathcal{B}_{\alpha\text{-}\rho}\subseteq \bigcup_c\mathcal{B}_{\alpha\text{-}\rhocm}$.
Suppose that~$\mathcal{G}$ is a graph class with bounded $\alpha\text{-}\rho$, say $\alpha\text{-}\rho(G)\le k$ for all $G\in \mathcal{G}$.
We claim that~$\mathcal{G}$ has bounded $\alpha\text{-}\mu_{\rho,0}$.
Consider an arbitrary graph $G\in \mathcal{G}$.
Then, there exists a set $S\subseteq V(G)$ such that $\alpha(G[S])\le k$ and $\eta(G-S)\le d$.
Since $\eta(G-S)\le d$, the graph $G-S$ satisfies $\rho(G-S) = 0$; consequently, $\alpha\text{-}\mu_{\rho,0}(G)\le \alpha(G[S])\le k$, showing that~$\mathcal{G}$ has bounded $\alpha\text{-}\mu_{\rho,0}$, as claimed.
Hence, $\mathcal{B}_{\alpha\text{-}\rho}\subseteq \mathcal{B}_{\alpha\text{-}\mu_{\rho,0}}\subseteq \bigcup_c\mathcal{B}_{\alpha\text{-}\rhocm}$.

We therefore have $\mathcal{B}_{\alpha\text{-}\rho} = \bigcup_{c}\mathcal{B}_{\alpha\text{-}\rhocm}$.
To conclude the proof, we show that $\mathcal{C}_{\rho} = \bigcup_{c}\mathcal{C}_{\rhocm}$.
By \Cref{modulator is awesome}, $\rho$ is awesome. 
Hence, the class properties $\mathcal{C}_{\rho}$ and 
$\mathcal{B}_{\alpha\text{-}\rho}$ coincide.
Furthermore, since $\rho$ is heavy, for each integer $c$, the class properties $\mathcal{C}_{\rhocm}$ and $\mathcal{B}_{\alpha\text{-}\rhocm}$ coincide, again by \Cref{modulator is awesome}.
This implies that $\bigcup_{c}\mathcal{C}_{\rhocm} = \bigcup_{c}\mathcal{B}_{\alpha\text{-}\rhocm}$ and, hence,
\[\mathcal{C}_{\rho} = \mathcal{B}_{\alpha\text{-}\rho} = \bigcup_{c}\mathcal{B}_{\alpha\text{-}\rhocm}
= \bigcup_{c}\mathcal{C}_{\rhocm}\,,\]
as claimed.
\end{proof}

However, as we show next, the inclusions $\bigcup_{c}\mathcal{B}_{\alpha\text{-}\rhocm}\subseteq \mathcal{B}_{\alpha\text{-}\rho}$ and $\bigcup_{c}\mathcal{C}_{\rhocm}\subseteq \mathcal{C}_{\rho}$ can be strict in general.

\begin{proposition}\label{prop:alpha-mu-rho-c-bounded-strict}
Let $\rho$ be an inheritable and heavy hyperparameter such that $\alpha\text{-}\rho$ is bounded on the class of disjoint unions of complete graphs.
Then, ${\bigcup_{c}\mathcal{B}_{\alpha\text{-}\rhocm}\subsetneq \mathcal{B}_{\alpha\text{-}\rho}}$.
\end{proposition}

\begin{proof}
Since $\rho$ is inheritable, \Cref{thm:alpha-rho at most alpha-rhocm} implies that ${\bigcup_{c}\mathcal{B}_{\alpha\text{-}\rhocm}\subseteq \mathcal{B}_{\alpha\text{-}\rho}}$.
Let~$\mathcal{G}$ be the class of disjoint unions of complete graphs.
By assumption,~$\mathcal{G}$ satisfies the class property $\mathcal{B}_{\alpha\text{-}\rho}$.
So it suffices to show that $\mathcal{G}\not\in \bigcup_{c}\mathcal{B}_{\alpha\text{-}\rhocm}$, that is, that for every integer $c$ the class~$\mathcal{G}$ has unbounded $\alpha\text{-}\rhocm$.
This means that for every integer $c$ and every positive integer $k$ there exists a graph $G\in \mathcal{G}$ such that $\alpha\text{-}\rhocm(G)\ge k$.
Since $\rho$ is heavy, it is unbounded on the class of complete graphs; hence, there exists a complete graph $K$ such that $\rho(K)>c$. 
Let $G$ be the disjoint union of $k$ copies of $K$.
Consider an arbitrary $(\rho,c)$-modulator $S$ of~$G$.
Then, for every component $C$ of~$G$, the set $S$ contains at least one vertex from $C$, since otherwise the fact that $\rho$ is monotone under induced subgraphs would imply $\rho(G-S)\ge \rho(C)>c$, a contradiction with the choice of $S$.
It follows that $\alpha(G[S])\ge k$.
Since $S$ was arbitrary, it follows that $\alpha\text{-}\rhocm(G)\ge k$, as desired.
\end{proof}

Examples of parameters satisfying the conditions of \Cref{prop:alpha-mu-rho-c-bounded-strict} include clique number, treewidth, pathwidth, and treedepth.

\begin{proposition}\label{prop:clique-bounded-strict}
Let $\rho$ be an inheritable and heavy hyperparameter such that $\mathcal{C}_{\rho} = \bigcup_{c}\mathcal{C}_{\rhocm}$.
Then, $\rho$ is awesome.
\end{proposition}

\begin{proof}
Since $\rho$ is heavy, by \cref{modulator is awesome}, the $(\rho,c)$-modulator number $\rhocm$ is awesome for every integer~$c$, that is, $\mathcal{C}_{\rhocm} = \mathcal{B}_{\alpha\text{-}\rhocm}$.
Consequently, $\bigcup_c\mathcal{C}_{\rhocm} = \bigcup_c\mathcal{B}_{\alpha\text{-}\rhocm}$.
Furthermore, since $\rho$ is inheritable, \Cref{thm:alpha-rho at most alpha-rhocm} implies that $\bigcup_c\mathcal{B}_{\alpha\text{-}\rhocm} \subseteq \mathcal{B}_{\alpha\text{-}\rho}$.
It follows that 
\[\mathcal{C}_{\rho} = \bigcup_{c}\mathcal{C}_{\rhocm}
= \bigcup_c\mathcal{B}_{\alpha\text{-}\rhocm} \subseteq \mathcal{B}_{\alpha\text{-}\rho}\,,\]
that is, $\rho$ is awesome.
\end{proof}

Observe that by \Cref{prop:clique-bounded-strict}, for every inheritable and heavy hyperparameter $\rho$ and integer~$c$, \Cref{prop:modulators-equality}  refines \Cref{modulator is awesome} for the  $(\rho,c)$-modulator number $\rhocm$.
Furthermore, since by \Cref{thm:alpha-rho at most alpha-rhocm}, every inheritable hyperparameter $\rho$ satisfies $\bigcup_{c}\mathcal{C}_{\rhocm}\subseteq \mathcal{C}_{\rho}$, \Cref{prop:clique-bounded-strict} has the following consequence.

\begin{corollary}\label{cor:clique-bounded-strict}
Let $\rho$ be an awful, inheritable, and heavy hyperparameter.
Then, $\bigcup_{c}\mathcal{C}_{\rhocm}\subsetneq \mathcal{C}_{\rho}$.
\end{corollary}

Examples of parameters satisfying the conditions of \Cref{cor:clique-bounded-strict} include treewidth, pathwidth, and treedepth (see \cref{cor:inheritable} along with \cref{tw-pw-awful} and~\cite{CT24}).

\subsection{Some applications of awesomeness}\label{sec:applications}

We now derive some consequences of \cref{cor:modulator-examples}, giving sufficient conditions for awesomeness. 

\subsubsection{Modulators to treewidth}

We start with the case when $\mu$ is a modulator to bounded treewidth, that is, $\mu=\twcm$ for some $c\in \mathbb{Z}$.
In this case, we obtain the following.

\begin{theorem}\label{thm:tw-modulator-omega-bdd}
Let $c \in \mathbb{Z}$, let $\mu$ be the $(\tw,c)$-modulator number, and let~$\mathcal{G}$ be a $(\mu,\omega)$-bounded graph class.
Then~$\mathcal{G}$ has bounded $\alpha$-treewidth.
\end{theorem}

\begin{proof}
\Cref{cor:modulator-examples} implies that~$\mathcal{G}$ has bounded $\alpha$-$\mu$, that is, there exists an integer $k$ such that for each graph $G\in \mathcal{G}$ and each induced subgraph~$G'$ of~$G$, there exists a set $S\subseteq V(G')$ such that $\alpha(G'[S])\le k$ and $\tw(G'-S)\le c$.
We can apply \cref{inheritable-lambda-rho at most} with $\rho = \tw$, since $\tw$ is strongly inheritable by \cref{obs:inheritable}, and $\lambda = \alpha$.
This implies that $\alpha$-$\tw(G) \leq \tw(G - S) + \alpha(G[S]) \leq c+k$, and thus shows that~$\mathcal{G}$ has bounded $\alpha$-treewidth, as claimed.
\end{proof}

\Cref{thm:tw-modulator-omega-bdd} implies that any $(\twcm,\omega)$-bounded graph class enjoys the good algorithmic properties of graph classes with bounded $\alpha$-treewidth.
In particular, Lima et al.~\cite{LMMORS24} showed that for any such graph class~$\mathcal{G}$, the following problems are solvable in polynomial time when restricted to graphs $G\in \mathcal{G}$:
\begin{itemize}
    \item For every $k,r$ and a $\mathsf{CMSO}_2$ formula $\psi$, given a graph $G$ and a weight function on the vertices of~$G$, find a maximum weight set $X \subseteq V(G)$ such that $G[X] \models \psi$ and $\omega(G[X]) \leq r$, or conclude that no such set exists.
    \item  For a fixed positive even integer $d$, given a collection of connected subgraphs of~$G$, each equipped with a real weight, find a maximum weight subcollection of the subgraphs any two of which are at  distance at least $d$ from each other.
\end{itemize}
Both of the above results capture the \textsc{Maximum Weight Independent Set} and the \textsc{Maximum Weight Induced Matching} problems (also studied for graphs with bounded $\alpha$-treewidth in~\cite{dallard2022firstpaper,MR3775804}).
Furthermore, the $\mathsf{CMSO}_2$-based problem listed above encompasses many other problems, including \textsc{Graph Homomorphism} to a fixed target graph, \textsc{Maximum Induced Forest} (which is equivalent to~\textsc{Minimum Feedback Vertex Set}), \textsc{Maximum Induced Bipartite Subgraph} (which is equivalent to~\textsc{Minimum Odd Cycle Transversal}), \textsc{Maximum Induced Planar Subgraph}, and \textsc{Maximum Induced Cycle Packing}.

In particular, \Cref{thm:tw-modulator-omega-bdd} and the polynomial-time solvability of the \textsc{Maximum Weight Independent Set} problem for graph classes with bounded $\alpha$-treewidth provide a partial answer to the following question.
\begin{question}[Dallard et al.~{\cite[Question 8.3]{dallard2022secondpaper}}; see also~{\cite[Question 9.8]
{MR4334541}}]
\label{question:alpha-easy}
Is it true that for every $(\tw,\omega)$-bounded graph class~$\mathcal{G}$, there exists a polynomial-time algorithm for the \textsc{Maximum Weight Independent Set} problem when restricted to graphs in~$\mathcal{G}$?
\end{question}

\noindent
Indeed, since treewidth is inheritable due to \cref{cor:inheritable}, we derive from \cref{thm:alpha-rho at most alpha-rhocm} \ref{rho-prop1} that, for any integer $c$, any graph class with bounded $(\tw,c)$-modulator number has bounded treewidth. Therefore, 
the answer to \cref{question:alpha-easy} is affirmative if the treewidth is replaced by any $(\tw,c)$-modulator number. 
Formally, we have the following result, providing an infinite sequence of graph parameters whose clique-boundedness implies polynomial-time solvability of the \textsc{Maximum Weight Independent Set} problem.

\begin{theorem}\label{thm:MWIS}
For every integer $c$ and every $(\mu,\omega)$-bounded graph class~$\mathcal{G}$, where $\mu$ is the $(\tw,c)$-modulator number, there exists a polynomial-time algorithm for the \textsc{Maximum Weight Independent Set} problem when restricted to graphs in~$\mathcal{G}$.
\end{theorem}

Recall that for $c\in \{1,2\}$, the $(\tw,c)$-modulator number coincides with the vertex cover number and the feedback vertex set number, respectively.
The result of \Cref{thm:MWIS} for $c = 1$ is not new, as for every $(\vc,\omega)$-bounded graph class~$\mathcal{G}$ there exists an integer $n$ such that every graph in~$\mathcal{G}$ is $nK_2$-free (see, e.g., \cref{nK2-Knn}), and this is known to imply polynomial-time solvability of the \textsc{Maximum Weight Independent Set} problem (see, e.g.,~\cite{MR3695266}).
On the other hand, to the best of our knowledge, the result of \Cref{thm:MWIS} is new for all $c\ge 2$.
In particular, for the case $c = 2$, we obtain the following.

\begin{corollary}\label{cor:fvs}
For every $(\fvs,\omega)$-bounded graph class~$\mathcal{G}$, the \textsc{Maximum Weight Independent Set} problem can be solved in polynomial time when restricted to graphs in~$\mathcal{G}$.
\end{corollary}

We next complement the case $c = 1$ of \Cref{thm:tw-modulator-omega-bdd} with a characterisation in terms of forbidden induced subgraphs.

\subsubsection{Vertex cover number}

The fact that the vertex cover number is awesome leads to a characterisation of hereditary graph classes with bounded $\alpha$-vertex cover number.
To derive this result, we recall the following Ramsey-type result.

\begin{theorem}[Curticapean and Marx~\cite{DBLP:conf/focs/CurticapeanM14},  
Lozin~\cite{MR4508171}]\label{thm:matchings}
For any three positive integers $p$, $q$, and $r$ there exists an integer $s$ such that any graph containing a matching of size $s$ contains either $K_p$ or $K_{q,q}$ or $rK_2$ as an induced subgraph.
\end{theorem}
 
\begin{theorem}\label{nK2-Knn}
Let~$\mathcal{G}$ be a hereditary graph class.
Then, the following conditions are equivalent.
\begin{enumerate}
    \item~$\mathcal{G}$ is $(\vc,\omega)$-bounded.
    \item~$\mathcal{G}$ has bounded $\alpha$-vertex cover number.
    \item There exists a positive integer $n$ such that~$\mathcal{G}$ excludes the graphs $nK_2$ and $K_{n,n}$.
\end{enumerate}
\end{theorem}

\begin{proof}
\Cref{cor:awesome-examples} implies the equivalence of the first two conditions.

Assume now that the first condition holds, that is,~$\mathcal{G}$ is $(\vc,\omega)$-bounded, and let $f:\mathbb{N}\to \mathbb{N}$ be a nondecreasing function such that for every graph $G$ in the class~$\mathcal{G}$, every induced subgraph~$G'$ of~$G$ satisfies $\vc(G')\le f(\omega(G'))$.
For every positive integer $n$, the graphs $nK_2$ and $K_{n,n}$ have clique number~$2$, and vertex cover number $n$.
Therefore,~$\mathcal{G}$ excludes the graphs $nK_2$ and $K_{n,n}$ for all $n\ge f(2)+1$.
Hence, the third condition holds.

Finally, assume that the third condition holds, that is, that there exists a positive integer $n$ such that~$\mathcal{G}$ excludes the graphs $nK_2$ and $K_{n,n}$.
By \cref{thm:matchings}, for any three integers $p$, $q$, and $r$, there exists a smallest integer $s(p,q,r)$ such that every graph containing a matching of size $s(p,q,r)$ contains either $K_p$ or $K_{q,q}$ or $rK_2$ as an induced subgraph.
We show that~$\mathcal{G}$ is $(\vc,\omega)$-bounded, more precisely, that every graph $G\in \mathcal{G}$ has a vertex cover with cardinality at most $f(\omega(G))$ where $f:\mathbb{N}\to \mathbb{N}$ is the function defined by $f(k) = 2(s(k+1,n,n)-1)$, for all $k\in \mathbb{N}$.
Let $G$ be a graph in~$\mathcal{G}$ and let $k$ be the clique number of~$G$.
Since $G$ has no clique of size $k+1$ nor an induced subgraph isomorphic to $K_{n,n}$ or $nK_2$, \cref{thm:matchings} implies that $G$ does not contain a matching of size $s(k+1,n,n)$.
Let $M$ be a maximum matching in $G$ and let $S$ be the set of vertices of~$G$ saturated by $M$.
Then $S$ is a vertex cover in $G$ with cardinality at most $2(s(k+1,n,n)-1)$.
Therefore, $\vc(G)\le f(k)$, as claimed.
\end{proof}

Since the independence number $\alpha$ is weakly submodular and tame, \Cref{prop:eta,prop:lambda-td-vc} imply that bounded $\alpha$-vertex cover number implies bounded $\alpha$-treewidth.
Hence, \Cref{nK2-Knn} provides a strengthening of the fact that, for any positive integer $n$, the class of $\{nK_2,K_{n,n}\}$-free graphs has bounded $\alpha$-treewidth.
The latter is a consequence of a more general result due to Abrishami et al.~\cite[Theorem~1.1]{abrishami2024excludingcliquebicliquegraphs}, as well as of a result of Hajebi and Spirkl~\cite[Theorem~1.4]{hajebi2026treealphaexcludingfinitelygraphs}.

\subsubsection{Odd cycle transversal number}

We next explain an algorithmic implication of \Cref{cor:awesome-examples} (and its proof) in relation with \textsc{Minimum Odd Cycle Transversal}, the problem of computing a minimum odd cycle transversal in a given graph $G$.
In \cref{thm:MWIS-oct}, we generalise the fact that \textsc{Maximum Weight Independent Set} is solvable in polynomial time in any graph class with bounded odd cycle transversal number (as can be easily seen; see, e.g., the concluding remarks of~\cite{MR3197780}) or, more generally, in any class of graphs whose vertex set admits a partition into a bounded number of cliques and two independent sets, as shown by Alekseev and Lozin (see~\cite{ALEKSEEV2003351}).

\begin{lemma}\label{lem:MWIS-finding-oct}
For every positive integer $k$, there exists an algorithm that, given a graph $G$ that admits an odd cycle transversal $S$ with $\alpha(G[S])\le k$, computes such a set $S$ in polynomial time.
\end{lemma}

\begin{proof}
Let $S$ be an odd cycle transversal in $G$ with independence number at most $k$. Then $G$ admits a partition of its vertex set into two parts $S$ and $V(G) \setminus S$ such that $G[S]$ has independence number at most $k$ and $G-S$ is bipartite. 
It follows from~\cite[Lemma~8]{ALEKSEEV2003351} and its proof that, given a graph $G$ satisfying this property, a witnessing vertex partition  can be computed in $\mathcal{O}(|V(G)|^{6k+2})$.
\end{proof}

\begin{theorem}\label{thm:MWIS-oct}
For every $(\octt,\omega)$-bounded graph class~$\mathcal{G}$, the \textsc{Maximum Weight Independent Set} problem can be solved in polynomial time when restricted to graphs in~$\mathcal{G}$.

\end{theorem}

\begin{proof}
Fix~$\mathcal{G}$ as in the theorem and let $G = (V,E)$ be a graph in~$\mathcal{G}$ given together with some vertex weight function.
We describe a polynomial-time algorithm for computing a maximum weight independent set in~$G$.
By \cref{cor:awesome-examples}, there exists a constant $k = k_\mathcal{G}$ independent of~$G$ such that $G$ admits an odd cycle transversal $S$ with $\alpha(G[S])\le k$.
Furthermore, by \cref{lem:MWIS-finding-oct}, such an odd cycle transversal $S$ can be computed in polynomial time.
For each independent set $I\subseteq S$, we compute a maximum weight independent set $I'$ in the bipartite graph $G[V\setminus S]-N(I)$.
Finally, among all such sets $I$, we return the set $I\cup I'$ with maximum weight.

For each of the $\mathcal{O}(|V(G)|^k)$ choices of $I$, the algorithm solves an instance of the \textsc{Maximum Weight Independent Set} problem in a bipartite graph, which can be reduced in polynomial time to solving an instance of the maximum flow problem (see~\cite{MR712925}).
Hence, the algorithm runs in polynomial time.

To see that the set returned by the algorithm is a maximum weight independent set in $G$, note that for any maximum weight independent set $J$ in $G$, writing $J = I \cup I'$ where $I = J\cap S$ and $I' = J\setminus S$, the set $I'$ is a maximum weight independent set in the bipartite graph $G[V\setminus S]-N(I)$.
Hence, a set of the same weight as $J$ will be detected by the algorithm in the iteration considering the independent set $I = J\cap S$.
\end{proof}

Since the feedback vertex set number of a graph is an upper bound on its odd cycle transversal number, \Cref{thm:MWIS-oct} generalises \Cref{cor:fvs}.
Furthermore, \Cref{thm:MWIS-oct} applies in a context beyond \Cref{question:alpha-easy}, since it provides a graph parameter $\mu$ that is incomparable with treewidth and such that clique-bounded $\mu$ implies polynomial-time solvability of the \textsc{Maximum Weight Independent Set} problem.

\subsubsection{Tree-independence number}

We conclude this section by applying \cref{lem:tau-k-omega-alpha-general} to establish logarithmic tree-independence number in certain classes of graphs.
Two vertex-disjoint cycles $C$ and $C'$ in a graph $G$ are said to be \emph{independent} if there is no edge between $C$ and $C'$. 
A graph is said to be \emph{$\mathcal{O}_k$-free} if it does not contain $k$ pairwise independent induced cycles, or equivalently, if it does not contain $k$ pairwise independent cycles.
Bonamy et al.~\cite{MR4723425} showed that for fixed $k$ and $\ell$, every $\mathcal{O}_k$-free graph not containing $K_{\ell,\ell}$ as a subgraph has logarithmic treewidth.
The following theorem provides a dense analogue of this result and provides an affirmative answer to \cite[Question 7.3]{HMV25}.

\begin{theorem}\label{thm:OK-free}
For every two positive integers $k$ and $t$ there exists a positive integer $c_{k,t}$ such that if $G$ is a $K_{t,t}$-free $\mathcal{O}_k$-free graph with at least two vertices, then $\tin(G)\le c_{k,t}\cdot \log |V(G)|+1$.
\end{theorem}

\begin{proof}
Recall that for a graph $G$, a feedback vertex set in $G$ is a set $S\subseteq V(G)$ such that $G-S$ is acyclic, and the feedback vertex set number of~$G$, denoted by $\fvs(G)$, is the minimum cardinality of a feedback vertex set in $G$.
As shown by Bonamy et al.~\cite{MR4723425}, for every positive integer $k$ there exists a nondecreasing function $b_k \colon \mathbb{N} \rightarrow \mathbb{N}$ such that for every positive integer $\ell$, every $\mathcal{O}_k$-free graph $G$ with at least two vertices not containing $K_{\ell,\ell}$ as a subgraph satisfies $\fvs(G)\le b_k(\ell)\cdot \log |V(G)|$.
Fix positive integers $k$ and $t$, and let~$\mathcal{G}$ be the class of $K_{t,t}$-free $\mathcal{O}_k$-free graphs.
Since every graph in~$\mathcal{G}$ is $K_{t,t}$-free, Ramsey's theorem guarantees the existence of a nondecreasing function $g_t$ such that if $G\in \mathcal{G}$, then $G$ does not contain $K_{\ell,\ell}$ as a subgraph where $\ell = g_t(\omega(G))$.
Combining this with the aforementioned result of Bonamy et al., we obtain that if $G\in  \mathcal{G}$ has at least two vertices, then $\fvs(G)\le b_k(g_t(\omega(G)))\cdot \log |V(G)|$.
Hence, taking $g(x)=\log(x)$ for $x \geq 2$ and $g(1) = 1$, $\rho=\tw$, $c = 2$, $h(x) = x$, and letting $f$ be the composition of $b_k$ and $g_t$, we obtain that every graph $G\in \mathcal{G}$ satisfies $\rhocm(G)\le f(\omega(G))\cdot g(|V(G)|)$.
Hence, by \cref{lem:tau-k-omega-alpha-general},  for every graph $G\in \mathcal{G}$, every minimum feedback vertex set $S$ in $G$  satisfies $\alpha(G[S])\le f(3)\cdot g(|V(G)|)$.
The graph $G-S$ is a forest, and hence, admits a tree decomposition $\mathcal T$ in which every bag is a clique; adding $S$ to every bag of $\mathcal T$ implies that $\tin(G) \le f(3)\cdot g(|V(G)|)+1$. 
Recalling the definition of~$G$ and since $\alpha$-treewidth of a single-vertex graph is 1, we can conclude that $\tin(G)\le f(3)\cdot \log |V(G)| + 1 = b_k(g_t(3))\cdot \log |V(G)| + 1$. Taking $c_{k,t} = b_k(g_t(3))$ proves the claim.
\end{proof}

\Cref{thm:OK-free} implies that for any two positive integers $k$ and $t$, the \textsc{Maximum Weight Independent Set} problem can be solved in quasipolynomial time when restricted to $K_{t,t}$-free $\mathcal{O}_k$-free graphs.
Moreover, following similar arguments as in~\cite{MR4955546,chudnovsky2025treeindependencenumberivsoda}, the same is true for many other natural algorithmic problems that are \textsf{NP}-hard in general.
 
\section{Discussion and open problems}\label{sec:conclusion}

Recall that Chudnovsky and Trotignon (see~\cite{CT24}) constructed graph classes with clique-bounded treewidth and unbounded $\alpha$-treewidth, thus disproving a conjecture of Dallard, Milani\v{c}, and \v{S}torgel~\cite{dallard2022secondpaper} that treewidth is awesome.
However, their construction says nothing about the following weaker implication.

\begin{question}\label{question-pw}
Is it true that every $(\pw,\omega)$-bounded graph class has bounded \hbox{$\alpha$-treewidth}?
\end{question}

As we explain next, the analogue of \Cref{question-pw} obtained by replacing pathwidth with treedepth has a positive answer.
This is a consequence of a recent result due to Hajebi and Spirkl~\cite{hajebi2026treealphaexcludingfinitelygraphs} stating that for any \textsl{finite} set~$\mathcal{F}$ of graphs, the class of $\mathcal{F}$-free graphs has bounded tree-independence number if and only if~$\mathcal{F}$ contains a complete bipartite graph, a graph each component of which is a path or a subdivided claw, and the line graph of such a graph.
Combined with the analogous characterisation for bounded treewidth due to Lozin and Razgon~\cite{MR4385180}, the result of Hajebi and Spirkl implies that treewidth is awesome \textsl{in a finite sense}, that is, every hereditary graph class defined by finitely many forbidden induced subgraphs that is $(\tw,\omega)$-bounded has bounded $\alpha$-treewidth.
It also implies the following sufficient condition for bounded $\alpha$-treewidth (see also~\cite[Theorem 1.2]{nikabadi2026inducedpackingtreewidthii} for a recent further strengthening).

\begin{theorem}\label{thm:Pn_Knn}
For every positive integer $n$, the class of $\{P_n,K_{n,n}\}$-free graphs has bounded $\alpha$-treewidth.
\end{theorem}

The announced result now follows easily.

\begin{theorem}\label{thm:td}
Every $(\td,\omega)$-bounded graph class has bounded $\alpha$-treewidth.
\end{theorem}

\begin{proof}
By \Cref{thm:Pn_Knn}, it suffices to show that every hereditary $(\td,\omega)$-bounded graph class excludes some path and some complete bipartite graph.
Let~$\mathcal{G}$ be a hereditary $(\td,\omega)$-bounded graph class, and let $f:\mathbb{N}\to \mathbb{N}$ be a nondecreasing function such that $\td(G)\le f(\omega(G))$ for all $G\in \mathcal{G}$.
Let $P_n$ be an $n$-vertex path that belongs to~$\mathcal{G}$.
Then $\td(P_n)\le f(2)$, since $\omega(P_n)\le 2$ and $f$ is nondecreasing.
However, $\td(P_n) = \lceil\log_2(n+1)\rceil$ (see, e.g.,~\cite{MR2920058}) and therefore $n\le 2^{f(2)}-1$.
Hence~$\mathcal{G}$ excludes the path of order $2^{f(2)}$.
The argument for complete bipartite graphs $K_{n,n}$ is similar, using the fact that they have bounded clique number but unbounded treedepth (note that $\td(K_{n,n})\ge \td(P_{2n})$, since $P_{2n}$ is a subgraph of~$K_{n,n}$).
\end{proof}

As a side remark, let us note that the conclusion of \Cref{thm:Pn_Knn} cannot be strengthened to bounded $\alpha$-pathwidth. 
Indeed, our construction of $(\td,\omega)$-bounded graph classes with unbounded $\alpha$-pathwidth given in \Cref{sec:bddtd-not-alpha-bdd-pw} results in $\{P_6,K_{2,2}\}$-free graphs.
In particular, this shows that treedepth and pathwidth are not awesome even in the aforementioned finite sense.
On the other hand, a strengthening to bounded $\alpha$-treedepth can be obtained if, instead of excluding a complete bipartite graph, we exclude a star $K_{1,n}$.
Indeed, the proof of \cite[Theorem 1.5]{DKKMMSW2024}, stating that the class of $\{P_n, K_{1,n}\}$-free graphs has bounded $\alpha$-treewidth, can be easily adapted to show that the class has bounded $\alpha$-treedepth (see  also~\cite{choi2025excludinginducedwheelminor} for a different proof).

\medskip
We conclude the paper with several further questions left open by this work.

A natural direction is to develop a better understanding of the landscape of awesome graph (hyper)parameters, and in particular to identify which parameters beyond those captured by \cref{modulator is awesome} are awesome.
For instance, it would be interesting to determine whether degeneracy is awesome or not, with respect to its canonical hyperparameterisation. A similarly defined max-min parameter, which we deem worthy of investigation with respect to awesomeness, is \emph{symmetric difference} of graphs.
Symmetric difference was implicitly introduced in \cite{ACLZ15}, and studied in \cite{AAL21,DLMSZ24,BDSZ24}.
It can be viewed as a dense analogue of degeneracy, and is formally defined for graphs $G$ with at least two vertices as the maximum, over all induced subgraphs $H$ of~$G$ with at least two vertices, of the minimum, over all pairs of distinct vertices $x,y$ in $H$, of the number of vertices distinct from $x$ and $y$ that are adjacent to exactly one of $x$ and $y$.

In \Cref{sec:td-pw}, we introduced a graph transformation that increases $\alpha$-pathwidth by exactly one, and the same for $\alpha$-treedepth.
We are not aware of any such transformation for $\alpha$-treewidth.

\begin{question}
Is there an efficiently computable graph transformation that increases $\alpha$-treewidth by a fixed constant?
\end{question}

Recall that by \Cref{prop:main}, any graph class~$\mathcal{G}$ that has bounded $\alpha\text{-}\rho$, where $\rho$ is a basic hyperparameter, has clique-bounded $\rho$ with a \textsl{polynomial} binding function.
This motivates the study of the following variant of awesomeness.
Let us say that a basic hyperparameter $\rho$ is \emph{weakly awesome} if polynomial clique-boundedness of $\rho$ is equivalent to boundedness of $\alpha\text{-}\rho$.
Of course, every awesome basic hyperparameter is also weakly awesome.
But there may be more.
 
\begin{question}
Which basic hyperparameters are weakly awesome? 
\end{question}

At the moment, we are in fact not aware of any weakly awesome basic hyperparameter that is not awesome: all the awful basic hyperparameters we are aware of (treewidth, pathwidth, treedepth, and chromatic number) turn out not to be weakly awesome, as we explain next.

In~\cite{CT24}, Chudnovsky and Trotignon constructed graph classes with polynomially clique-bounded treewidth but unbounded $\alpha$-treewidth.
This shows that treewidth is not weakly awesome.
Furthermore, our proof of \Cref{tw-pw-awful}, being based on a graph class in which the treedepth is bounded by a linear function of the clique number, shows that pathwidth and treedepth are not weakly awesome.

The argument based on \Cref{prop:main} and~\cite{MR4707561} (recall the discussion following \Cref{main-question} on p.~\pageref{discussion-awful-chi}) ruling out awesomeness of any hyperparameterisation of the chromatic number does not say anything about weak awesomeness.
Nonetheless, weak awesomeness of the canonical hyperparameterisation of the chromatic number given in \Cref{sec:hyperparameterisations} can be ruled out with the following much simpler argument.
The family of disjoint unions of complete graphs forms a $\chi$-bounded graph class~$\mathcal{G}$ that has unbounded $\alpha$-chromatic number, since for every $n\in \mathbb{N}$ and every proper colouring of the graph $nK_n\in \mathcal{G}$, there exists an independent set of $n$ vertices with pairwise different colours.

We leave for future research the study of various other hyperparameterisations of the chromatic number, including the question of whether the parameter admits a weakly awesome hyperparameterisation.

\paragraph{Acknowledgements}

The authors are grateful to Ond\v{r}ej Such\'y for helpful discussions regarding \Cref{rem:suchy}.
This work is supported in part by the Slovenian Research and Innovation Agency (I0-0035, research programs P1-0285 and P1-0383, and research projects J1-60012, J1-70035, J1-70046, and N1-0370).

\printbibliography

\end{document}